\documentclass[12pt]{article}
\usepackage{amsmath}
\usepackage{amsfonts}
\usepackage{amssymb}
\usepackage{theorem}
\usepackage{xcolor}
\usepackage{float}
\usepackage{soul}
\usepackage{graphicx}
\usepackage{epstopdf}
\usepackage{caption}
\usepackage{subcaption}
\usepackage{enumerate}

\definecolor{labelkey}{rgb}{0,0.08,0.45}
\definecolor{refkey}{rgb}{0,0.6,0.0}
\definecolor{Brown}{rgb}{0.45,0.0,0.05}
\definecolor{lime}{rgb}{0.00,0.8,0.0}
\definecolor{lblue}{rgb}{0.5,0.5,0.99}

 \usepackage{mathpazo}

\providecommand{\xso}{x_1^*}
\providecommand{\xst}{x_2^*}

\newcommand{\To}{\ensuremath{\rightrightarrows}}

\newcommand{\scal}[2]{\langle{{#1},{#2}}\rangle}

\newcommand{\RR}{\ensuremath{\mathbb R}}

\newcommand{\RX}{\ensuremath{\,\left]-\infty,+\infty\right]}}

\newcommand{\NN}{\ensuremath{\mathbb N}}

\newcommand{\menge}[2]{\big\{{#1} \mid {#2}\big\}}

\newcommand{\dom}{\ensuremath{\operatorname{dom}}}
\newcommand{\closu}{\ensuremath{\operatorname{cl}}}

\newcommand{\gra}{\ensuremath{\operatorname{gra}}}

\newcommand{\inte}{\ensuremath{\operatorname{int}}}

\newcommand{\reli}{\ensuremath{\operatorname{ri}}}

\newcommand{\aff}{\ensuremath{\operatorname{aff}}}
\newcommand{\suchthat}{\ensuremath{\;|\;}}

\newcommand{\ran}{\ensuremath{\operatorname{ran}}}
\newcommand{\rec}{\ensuremath{\operatorname{rec}}}

\newcommand{\conv}{\ensuremath{\operatorname{conv}}}

\renewcommand{\iff}{\ensuremath{\Leftrightarrow}}
\renewcommand{\phi}{\ensuremath{\varphi}}

\newcommand{\bx}{\ensuremath{\mathbf{x}}}

\providecommand{\abs}[1]{\lvert#1\rvert}

\providecommand{\stb}[1]{\left\{#1\right\}}

\providecommand{\LA}{\Leftarrow}
\providecommand{\RA}{\Rightarrow}

\providecommand{\grad}{\nabla}

\providecommand{\eps}{\epsilon}

\providecommand{\lam}{\lambda}
\providecommand{\RR}{\mathbb{R}}

\providecommand{\aff}{\operatorname{aff}}
\providecommand{\conv}{\operatorname{conv}}
\providecommand{\ran}{\operatorname{ran}}

\providecommand{\intr}{\operatorname{int}}

\providecommand{\dom}{\operatorname{dom}}

\providecommand{\gra}{\operatorname{gra}}

\providecommand{\fady}{\varnothing}

\providecommand{\rras}{\rightrightarrows}
\providecommand{\To}{\rightrightarrows}

\providecommand{\NN}{\mathbb{N}}

\providecommand{\ran}{\operatorname{ran}}
\providecommand{\rec}{\operatorname{rec}}
\providecommand{\core}{\operatorname{core}}

\providecommand{\pt}{{\partial}}

\providecommand{\spn}{\operatorname{span}}

\providecommand{\rra}{\rightrightarrows}

\providecommand{\fady}{\varnothing}

\providecommand{\ri}{\operatorname{ri}}

\providecommand{\RR}{\mathbb{R}}
\providecommand{\NN}{\mathbb{N}}

\newtheorem{theorem}{Theorem}[section]
\newtheorem{lemma}[theorem]{Lemma}
\newtheorem{lem}[theorem]{Lemma}
\newtheorem{corollary}[theorem]{Corollary}
\newtheorem{cor}[theorem]{Corollary}
\newtheorem{proposition}[theorem]{Proposition}
\newtheorem{prop}[theorem]{Proposition}
\newtheorem{definition}[theorem]{Definition}

\newtheorem{thm}[theorem]{Theorem}%[section]
\theoremstyle{plain}{\theorembodyfont{\rmfamily}
}
\theoremstyle{plain}{\theorembodyfont{\rmfamily}
}
\theoremstyle{plain}{\theorembodyfont{\rmfamily}
}
\theoremstyle{plain}{\theorembodyfont{\rmfamily}
\newtheorem{example}[theorem]{Example}}

\newtheorem{fact}[theorem]{Fact}
\theoremstyle{plain}{\theorembodyfont{\rmfamily}
\newtheorem{remark}[theorem]{Remark}}

\newcommand{\pluss}{{\hskip1pt \raise1pt\vbox{\hrule width6pt \vskip1pt
\hrule width6pt}\kern-4pt{\lower1pt\hbox{\vrule height6pt \kern1pt\vrule
height6pt}}\hskip5pt}}

\begin{document}

\title{{\sffamily Nearly convex sets: fine properties and domains or ranges of subdifferentials of convex functions}}

\author{
Sarah M.\ Moffat\thanks{Mathematics, Irving K.\ Barber School, University
of British Columbia Okanagan, Kelowna, British Columbia V1V 1V7,
Canada. E-mail: \texttt{sarah.moffat@ubc.ca}.},
Walaa M.\ Moursi\thanks{Mathematics, Irving K.\ Barber School, University
of British Columbia Okanagan, Kelowna, British Columbia V1V 1V7,
Canada. E-mail: \texttt{walaa.moursi@ubc.ca}.},
and
Xianfu Wang\thanks{Mathematics, Irving K.\ Barber School, University
of British Columbia Okanagan, Kelowna, British Columbia V1V 1V7,
Canada. E-mail: \texttt{shawn.wang@ubc.ca}.}
\\
\\
Dedicated to R.T. Rockafellar on the occasion of his 80th birthday.
}

\date{\today}

\maketitle

% \vskip 8mm

\begin{abstract} \noindent
Nearly convex sets play important roles in convex analysis, optimization and theory of monotone operators.
We give a systematic study of nearly convex sets, and construct examples
of subdifferentials of lower semicontinuous convex functions whose domain
or ranges are nonconvex.
\end{abstract}



%%%%%%%%%%%%%%%%%%%%%%%%%%%%%%%%%%%%%%%%%%%%%%%%%%%%%%%%%%%%%%%%%%%%%%%%%%%%%%%%%%%%%%%%%%%%%%%%
%%%%%%%%%%%%%%%%%%%%%%%%%%%%%%%%%%%%%%%%%%%%%%%%%%%%%%%%%%%%%%%%%%%%%%%%%%%%%%%%%%%%%%%%%%%%%%%%
%%%%%%%%%%%%%%%%%%%%%%%%%%%%%%%%%%%%%%%%%%%%%%%%%%%%%%%%%%%%%%%%%%%%%%%%%%%%%%%%%%%%%%%%%%%%%%%%
\section{Introduction}
In 1960s Minty and Rockafellar coined nearly convex sets \cite{minty,Rockvirtual}.
Being a generalization of convex sets, the notion of near convexity or almost convexity has been gaining popularity in the optimization community, see \cite{bmw-nenc, Bot2006, Bot2007, Bot2008, Frenk}. This can be attributed to the applications of generalized convexity in economics problems, see for example, \cite{John, Martinez}.
One reason to study nearly convex sets is that
for a proper lower semicontinuous convex function
%$f:\RR^n\rightarrow\RX$, as $\dom f$ is convex and that
its subdifferential domain
%satisfies
%$$\reli\dom f\subseteq\dom\partial f\subseteq \dom f,$$
%$\dom \partial f$
is
always nearly convex \cite[Theorem 23.4, Theorem 6.1]{Rock70}, and the same is true for the domain of each maximally monotone operator \cite[Theorem 12.41]{Rock98}. Maximally monotone operators are extensively studies recently \cite{bartz07, bbw2007, bwy2012,
BC2011, Si2}. Another reason is that to study possibly nonconvex functions, a first endeavor perhaps should be to study functions
whose epigraphs are nearly convex, see, e.g., \cite{Bot2008}.

All these motivate our systematic study of nearly convex sets.
Some properties of nearly convex sets have been partially studied
in \cite{bmw-nenc, Bot2006, Bot2007,Bot2008} from different perspectives. The purpose of this paper is to
give new proofs to some known results, provide further characterizations, and extend known results on calculus, relative interiors,
recession cones, and applications.
Although nearly convex sets need not be convex, many results on
convex sets do extend.
We also construct proper lower semicontinuous convex functions whose subdifferential mappings have domains
being neither closed nor open; or highly nonconvex.

We remark that
nearly convex was called \emph{almost convex} in \cite{Bot2006, Bot2007, Bot2008}. Here, we adopt the term \emph{nearly convex}
rather than \emph{almost convex} because of the relationship with nearly equal sets which was noted in \cite{bmw-nenc}.
Note that this definition of nearly convex does not coincide with the one provided in \cite[Definition 2]{Bot2006} and \cite{Bot2007},
where nearly convex is a generalization of midpoint convexity.

The remainder of the paper is organized as follows. Some basic notations and facts about convex sets and nearly convex sets
are given in Section~\ref{s:facts}.
%we collect some facts
%used in sequel.
Section~\ref{s:characterization:basic} gives new characterizations of nearly convex sets. In Section~\ref{s:calculus:interior},
we give calculus of nearly convex sets and relative interiors.  In Section~\ref{s:recession:cone}, we study recessions
of nearly convex sets.
Section~\ref{s:application:max:further} is devoted to  apply results in Section~\ref{s:calculus:interior}
and Section~\ref{s:recession:cone} to study maximality of sum of several maximally monotone operators and closedness of nearly convex sets under a linear mapping.  In Section~\ref{s:range:domain}, we construct examples of proper lower semicontinuous
convex functions with prescribed nearly convex sets being their subdifferential domain. As early as 1970s, Rockafellar
provided a convex function whose subdifferential domain is not convex \cite{Rock70}. We give a detailed analysis of
his classical example and use it to generate new examples with pathological subdifferential domains.  Open problems appear in Section~\ref{s:openproblems}. Appendix A contains some proofs of Section~\ref{s:range:domain}.

\section{Preliminaries}\label{s:facts}
\subsection{Notation and terminology}
%we shall introduce some notation and concepts used in the paper.
Throughout this paper, we work in the Euclidean space $\RR^n$ with norm $\|\cdot\|$ and inner product $\langle \cdot, \cdot\rangle$.
For a set $C \subseteq \RR^n$ let $\closu{C}$ denote the closure of $C$, and $\aff C$ the \emph{affine hull} of $C$; that is, the smallest affine set containing $C$.
The key object we shall study is:
\begin{definition}[near convexity]
A set $E \subseteq \RR^n$ is \emph{nearly convex} if there exists a convex set $C \subseteq \RR^n$ such that
$$C \subseteq E \subseteq \closu{C}.$$
\end{definition}
Obviously, every convex set is nearly convex, but they are many nearly convex sets which are not convex. See Figure~\ref{fig:c:set} for two nearly convex sets.
\begin{figure}[H]
\begin{center}
\begin{tabular}{c c c}
%\hline
\includegraphics[scale=0.5]{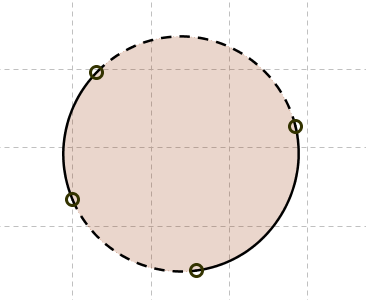}    & &
\includegraphics[scale=0.5]{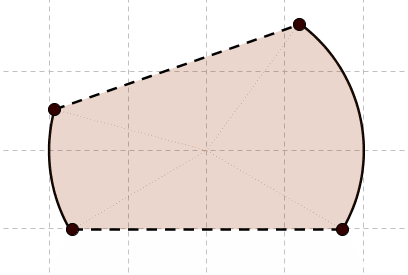}   \\
\end{tabular}
\end{center}
 \caption{A \texttt{GeoGebra} \cite{geogebra} snapshot.
Left: Neither open nor closed convex set. Right: Nearly convex but not convex set}
\label{fig:c:set}
\end{figure}
Note that nearly convex sets do not have nice algebra as convex sets do \cite[Section 3]{Rock70}, as
the following two simple examples illustrate.
\begin{example} The nearly convex set  $C\subseteq \RR^2$ given by
$$C:=\menge{(x_{1},x_{2})}{-1<x_{1}<1, x_{2}>0}\cup\{(-1,0), (1,0)\}.$$
has
$2C\neq C+C$, since
\begin{align*}
2C & =\menge{(x_{1},x_{2})}{-2<x_{1}<2, x_{2}>0}\cup\{(-2,0), (2,0)\},\\
C+C & =2C\cup{(0,0)}.
\end{align*}
On the contrary, $2C=C+C$ whenever $C$ is a convex set \cite[Theorem 3.2]{Rock70}.
\end{example}
\begin{example}
Define
$$E_{1}:=\menge{(x_{1},x_{2})}{x_{1}\geq 0, x_{2}\in\RR}\setminus\menge{(0,x_{2})}{|x_{2}|<1},$$
$$E_{2}:=\menge{(x_{1},x_{2})}{x_{1}\leq 0, x_{2}\in\RR}\setminus\menge{(0,x_{2})}{|x_{2}|<1}.$$
The set $E_{1}\cap E_{2}=\menge{(0,x_{2})}{|x_{2}|\geq 1}$ is not nearly convex.
On the contrary, $E_{1}\cap E_{2}$ is convex if both $E_{1}, E_{2}$ are convex.
\end{example}

Let $B(x,\varepsilon)\subset\RR^n$ be the closed ball with radius $\varepsilon>0$ and centered at $x$, and let
$I$ be an index set
$I:=\{1,2,\dots,m\}$
for some integer $m$.
We use $\conv C$ for the convex hull of $C$. The interior of $C$ is $\inte C$, the core is $$\core C:=\{ x \in C \suchthat (\forall y \in \RR^n)(\exists \varepsilon > 0) \left[x-\varepsilon y, x+\varepsilon y\right] \subseteq C\},$$ and the relative interior is $$\ri C:= \{ x \in \aff C \suchthat \exists \varepsilon>0, (x+\varepsilon B(0,1)) \cap (\aff C) \subseteq C\}.$$
The recession cone of $C$ is
\begin{equation}\label{eq:rec:d}
\rec C := \{ y \in \RR^n \suchthat (\forall \lam \geq 0)~\lam y+C\subseteq C\}.
\end{equation}
The lineality space of $C$ is the largest subspace contained in $\rec C$, see \cite[page 65]{Rock70} for more on lineality spaces.
We denote the projection operator onto the set $C\subseteq \RR^n$ by $P_C$ and the normal cone operator by $N_C$.

For a set-valued mapping $A: \RR^n \To \RR^n$, the domain is $\dom A := \{ x\in \RR^n \suchthat Ax \neq \varnothing\}$, the range is $\ran A := \bigcup_{x\in \RR^n} Ax,$ and the graph is $\gra A :=\{(x,u) \in \RR^n \times \RR^n \suchthat u \in Ax \}.$ $A$ is monotone if $(\forall (x,u) \in \gra A)(\forall (y,v)\in \gra A) \quad \scal{x-y}{u-v} \geq 0$, and maximally monotone
if there exists no monotone operator $B$ such that $\gra A$ is a proper subset of $\gra B$.

\subsection{Auxiliary results on convex sets}
Properties of convex sets play a prominent role in the paper, we need to review some key results.
\begin{fact}[Rockafellar] \label{f:ri}
Let $C$ and $D$ be convex subsets of $\RR^n$,
and let $\lambda \in \RR$.
Then the following hold:
\begin{enumerate}
\item
\label{f:ri:i--}
$\reli C$ and $\closu{C}$ are convex.
\item
\label{f:ri:i-}
$C\neq\varnothing$ $\Rightarrow$ $\reli C \neq\varnothing$.
\item
\label{f:ri:i}
$\closu{(\reli C)} = \closu{C}$.
\item
\label{f:ri:iii}
$\reli C = \reli(\closu{C})$.
\item
\label{f:ri:ii}
$\aff(\reli C) = \aff C = \aff(\closu{C})$.
\item
\label{f:ri:vi}
$\reli C = \reli D$
$\Leftrightarrow$
$\closu{C}= \closu{D}$
$\Leftrightarrow$
$\reli C \subseteq D \subseteq \closu{C}$.
\item \label{f:ri:scal} $\reli \lambda C=\lambda \reli C$.
\item \label{f:ri:sum}
$\reli(C+D)=\reli C+\reli D.$
\end{enumerate}
\end{fact}
\begin{proof}
\ref{f:ri:i--}\&\ref{f:ri:i-}: See \cite[Theorem~6.2]{Rock70}.
\ref{f:ri:i}\&\ref{f:ri:iii}:
See \cite[Theorem~6.3]{Rock70}.
\ref{f:ri:ii}: See \cite[Theorem~6.2]{Rock70}.
\ref{f:ri:vi}:
See \cite[Corollary~6.3.1]{Rock70}.
\ref{f:ri:scal}:
See \cite[Corollary~6.6.1]{Rock70}.
\ref{f:ri:sum}:
See \cite[Corollary~6.6.2]{Rock70}.
\end{proof}

\begin{fact}\emph{\cite[Theorem 6.5]{Rock70}}
\label{f:riintersec}
Let $C_i$ be a convex set in $\RR^n$ for $i=1, \ldots, m$  such that $\bigcap\limits_{i=1}^m \reli C_i \neq \varnothing$. Then
$$\closu{ \big(\bigcap\limits_{i=1}^m C_i\big)} = \bigcap\limits_{i=1}^m \closu{C_i},$$
and
$$\reli \big(\bigcap\limits_{i=1}^m C_i\big) = \bigcap\limits_{i=1}^m \reli C_i.$$
\end{fact}

\begin{fact}\emph{\cite[Theorem 6.1]{Rock70}}
\label{l:access}
Let $C$ be a convex set in $\RR^n$, $x\in \ri C$, and $y\in \closu{C}$. Then
$$[x,y[\subseteq \ri C.$$
\end{fact}

\begin{fact}\emph{\cite[Theorem 6.6]{Rock70}}\label{f:convlinear}
Let $C$ be a convex set in $\RR^n$ and let $A$ be a linear transformation from $\RR^n$ to $\RR^m$. Then
$$\reli (AC) = A (\reli C),$$
and
$$A(\closu{C}) \subseteq \closu{(AC)}.$$
\end{fact}

\begin{fact}\emph{\cite[Theorem 6.7]{Rock70}} \label{f:convlinearinv}
Let $C$ be a convex set in $\RR^n$ and let $A$ be a linear transformation from $\RR^n$ to $\RR^m$ such that $A^{-1}(\reli C) \neq \varnothing$. Then
$$\reli(A^{-1}C) = A^{-1}(\reli C),$$
and
$$\closu{({A^{-1}C})} = A^{-1}(\closu{C}).$$
\end{fact}

\begin{fact}\emph{\cite[Theorem~8.1 \& Theorem~8.2]{Rock70}}\label{f:recC}
Let $C$ be a nonempty convex subset in $\RR^n$.
Then $\rec C$ is a convex cone and $0\in \rec C$.
If in addition $C$ is closed then $\rec C$ is closed.
\end{fact}

\begin{fact}\emph{\cite[Theorem~8.3]{Rock70}}\label{f:recC:II}
Let $C$ be a nonempty convex subset in $\RR^n$.
Then $\rec (\ri C)=\rec (\closu{C})$.
\end{fact}

\begin{fact}\emph{\cite[Theorem~9.1]{Rock70}}\label{f:recC:CL:I}
Let $C$ be a nonempty convex subset in $\RR^n$
and let $A$ be a linear transformation from $\RR^n$ to $\RR^m$.
Suppose that $(\forall z\in \rec(\closu{C})\setminus\stb{0})$
with $Az=0$ we have that $z$ belongs to the lineality space of
$\closu{C}$. Then
$$\closu{(AC)}=A(\closu{C}),$$
and
$$\rec A(\closu{C})=A[\rec(\closu{C})].$$
\end{fact}

\begin{fact}\emph{\cite[Corollary~9.1.1]{Rock70}}\label{f:recC:Cl:II}
Let $(E_i)_{i\in I}$,  be a family of
nonempty convex subsets in $\RR^n$
satisfying the following condition:
if $(\forall i\in I)( \exists z_i\in \rec(\closu{E_i}))$
and $\sum_{i\in I} z_i=0$ then $(\forall i \in I) z_i$ belongs to the lineality space
of $\closu{E_i}$. Then
\begin{align}
\closu{(E_1+\cdots+E_m)}&=\closu{E_1}+\cdots+\closu{E_m},\\
\rec[\closu{(E_1+\cdots+E_m)}]&=\rec(\closu{E_1})+\cdots+\rec(\closu{E_m}).
\end{align}
\end{fact}

\subsection{Auxiliary results on nearly convex sets}
Near equality introduced in \cite{bmw-nenc} provides a convenient tool to study
nearly convex sets and ranges of maximally monotone operators.

\begin{definition}[near equality]
\label{d:ne}
Let $C$ and $D$ be subsets of $\RR^n$.
We say that $C$ and $D$ are \emph{nearly
equal}\index{nearly equal}, if
\begin{equation}
\closu{C} = \closu{D}\text{ and }\reli C = \reli D.
\end{equation}
and denote this by $C\approx D$.
\end{definition}

\begin{fact}\emph{\cite[Lemma 2.7]{bmw-nenc}}
\label{l:nearconset}
Let $E$ be a nearly convex subset of $\RR^n$, say
$C\subseteq E\subseteq \closu{C}$, where
$C$ is a convex subset of $\RR^n$.
Then
\begin{equation}
E\approx\closu{E}\approx\reli E\approx\conv E
\approx\reli\conv E \approx C.
\end{equation}
In particular, the following hold.
\begin{enumerate}
\item\label{o:c} $\closu{E}$ and $\reli E$ are convex.
\item\label{t:c} If $E\neq\varnothing$, then $\reli E\neq\varnothing$.
\end{enumerate}
\end{fact}

\begin{fact}\emph{\cite[Proposition~2.12(i),(ii),(iii)]{bmw-nenc}} \label{ri:part}
Let $E_1$ and $E_2$
be nearly convex subsets of $\RR^n$. Then
\begin{equation}
E_1\approx E_2 \iff \ri E_1= \ri E_2\iff\closu{E_1}=\closu{E_2}.
\end{equation}
\end{fact}

\begin{fact}\emph{\cite[Proposition~2.5]{bmw-nenc}} \label{sq:thm}
Let $A$, $B$ and $C$
be subsets of $\RR^n$ such that $A\approx C$ and $A\subseteq B\subseteq C$. Then
$A\approx B\approx C$.
\end{fact}

\begin{fact}
 {\rm(See \cite[Theorem~12.41]{Rock98}.)}
 \label{D-R-nc}
Let $A:\RR^n \rras \RR^n$ be maximally monotone.
Then $\dom A$ and $\ran A$
are nearly convex.
\end{fact}

\begin{remark}
\label{BH:local:op:cl} Fact~\ref{D-R-nc} can not be localized.
Suppose that $A:=P_{B(0,1)}$ is the projection onto the unit ball in $\RR^2$, which is
a gradient mapping of the continuous differentiable convex function $f:\RR^2\rightarrow \RR$
%As
%$e(\iota_{C})=d_{C}^2/2$, we have
%$\nabla e (\iota_{C})=\Id-P_{C}$, and
 given by
$$f(x):=\begin{cases}
\|x\|^2/2 & \text{ if $\|x\|\leq 1$,}\\
\|x\|-1/2 & \text{ if $\|x\|>1$.}
\end{cases}
$$
\begin{enumerate}
\item Let $S:=\menge{(x,y)\in\RR^2}{x+y>2, x>0, y>0}$ be open convex. The set
$\ran P_{B(0,1)}(S)=\menge{(x,y)\in \RR^2}{x^2+y^2=1, x>0, y>0}$ is not nearly convex.
\item Let  $S:=\menge{(x,y)\in\RR^2}{x+y\geq 2, x\geq 0, y\geq 0}$ be closed convex. The set
$\ran P_{B(0,1)}(S)=\menge{(x,y)\in \RR^2}{x^2+y^2=1, x\geq 0, y\geq 0}$ is not nearly convex.
%\item Let $S=(t,2-t)$ for $t\in [0,2]$. The set
%$A(S)= (\cos \theta,\sin \theta)$ for $\theta\in [0,\pi/2]$,
%is not nearly convex.
\end{enumerate}
\end{remark}
We refer readers to \cite{lewis, bzhu, BC2011, Rock70, Si2} for more materials on convex analysis and
monotone operators.

\section{Characterizations and basic properties of nearly convex sets}\label{s:characterization:basic}

Utilizing near equality, in \cite{bmw-nenc} the authors provide the following characterizations of
nearly convex sets.
\begin{fact}[characterization of near convexity]\emph{\cite[Lemma 2.9]{bmw-nenc}}
\label{l:ruger}
Let $E\subseteq \RR^n$.
Then the following are equivalent:
\begin{enumerate}
\item
\label{l:rugeri}
$E$ is nearly convex.
\item
\label{l:rugerii}
$E\approx\conv E$.
\item
\label{l:rugerii+}
$E$ is nearly equal to a convex set.
\item
\label{l:rugerii++}
$E$ is nearly equal to a nearly convex set.
\item
\label{l:rugeriii}
$\reli\conv E \subseteq E$.
\end{enumerate}
\end{fact}

%\hl{The following Theorem is added on June 27.}
We now provide further characterizations of nearly convex sets.
\begin{thm}\label{nc:charac}
Let $E$ be a nonempty subset of $\RR^n$.
Then the following are equivalent:
\begin{enumerate}
\item\label{t:nc:charac:1}
$E$ is nearly convex.
\item\label{t:nc:charac:2}
$\ri E$ is convex and $ \closu{(\ri E)}=\closu{E}$.
\item\label{t:nc:charac:3}
$(\forall x\in \ri E) $ $(\forall y\in E)$ $[x,y[\subseteq \ri E$.
\item\label{t:nc:charac:4}
$\closu{E}$ is convex and $\ri (\closu{ E})\subseteq\ri{E}$.
\end{enumerate}
\end{thm}
\begin{proof}
\ref{t:nc:charac:1}$\RA$\ref{t:nc:charac:2}:
This follows from Fact~\ref{l:nearconset}.
\ref{t:nc:charac:2}$\RA$\ref{t:nc:charac:1}:
We have $\ri E\subseteq E\subseteq \overline{E}=\closu{(\ri E)}$.
%hence $\ri E \approx E$.
Since $\ri E$ is convex,
%it follows
%from Fact~\ref{l:ruger}~\ref{l:rugerii+} that
$E$ is nearly convex.
\ref{t:nc:charac:2}$\RA$\ref{t:nc:charac:3}:
Since $E$ is nearly convex, by Fact~\ref{l:nearconset}\ref{o:c}
$\closu{E}$ is convex.
Using Fact~\ref{l:access} applied to the convex set $\closu{E}$
we conclude that
$(\forall x\in \ri E) $ $(\forall y\in \closu{E})$ $[x,y[\subseteq \ri (\closu{E})=\ri E$.
In particular, $(\forall x\in \ri E) $ $(\forall y\in E)$ $[x,y[\subseteq \ri E$.
\ref{t:nc:charac:3}$\RA$\ref{t:nc:charac:2}: Let $x\in \ri E$ and $y\in \ri E$.
Then $[x,y]=[x,y[~\cup \stb{y}\subseteq \ri E \cup \ri E=\ri E$, that is
$\ri E$ is convex. It is obvious that $\closu{(\ri E) }\subseteq \closu{E }$.
Now we show that $\closu{E }\subseteq \closu{(\ri E )}$. Indeed, let $y\in \closu{E }$.
Then $(\exists (y_n)_{n\in \NN})\subseteq E$ such that $y_n\to y$.
Therefore $(\forall x\in \ri E)$ $[x,y_n[\subseteq \ri E $, and consequently
$[x,y_n]\subseteq \closu{(\ri E)} $.
That is, $(y_n)_{n\in \NN}\subseteq \closu(\ri E)$, hence
$y\in \closu{(\ri E)}$,
and therefore $\closu{E }\subseteq \closu{(\ri E )}$, as claimed.
\ref{t:nc:charac:1}$\RA$\ref{t:nc:charac:4}:
This follows from Fact~\ref{l:nearconset}.
\ref{t:nc:charac:4}$\RA$\ref{t:nc:charac:1}:
Since $E$ is nonempty we have $\closu{E} $
is nonempty and convex by assumption, hence
	$\closu{[\ri (\closu{E})]}=\closu{E}$.
Now $\ri(\closu{ E})\subseteq \ri E\subseteq E
\subseteq \closu{E}=\closu{[\ri (\closu{E})]}$,
hence $\ri (\closu{E}) \approx E$.
Since $\closu{E}$ is convex
we have $\ri (\closu{E})$ is convex. It follows
from Fact~\ref{l:ruger}\ref{l:rugerii+} that $E$ is nearly convex.
\end{proof}

\begin{theorem}
Let $E$ be a nonempty subset in $\RR^n$. Then
$E$ is nearly convex if and only if $E=C\cup S$ where
$C$ is a nonempty convex subset of $\RR^n$ and
$S\subseteq \closu{C}\setminus \ri C$.
\end{theorem}
\begin{proof}
$(\RA)$ Suppose that $E$ is nearly convex,
and notice that $E=\ri E\cup (E\setminus \ri E)$.
Set
\begin{equation}
C:=\ri E\;\;\;\text{and}\;\;\; S:=E\setminus \ri E\subseteq \closu{E}\setminus \ri E.
\end{equation}
Since $E$ is nearly convex, it follows from
Fact~\ref{l:nearconset} that $C$ is nonempty and convex.
Moreover $\closu{E}=\closu{(\ri E)}=\closu{C}$.
Therefore $E=C\cup S$ with $C$ convex and
 $S\subseteq \closu{C}\setminus \ri C$.
 $(\LA)$ Conversely, assume that $E=C\cup S$ where
$C$ is a nonempty convex subset of $\RR^n$ and
$S\subseteq \closu{C}\setminus \ri C$.
Clearly, $S\subseteq \pt (\closu{C})$,
where $\pt (\closu{C})$
is the relative boundary of $\closu{C}$.
Hence $\ri E=\ri C$ and consequently
$\ri E$ is nonempty and convex.
Moreover, since $\closu{S}\subseteq \closu{C}$, we have
$\closu{E}=\closu{(C\cup S)}=\closu{C}\cup \closu{S}=\closu{C}$
and consequently
$\closu{E}$ is convex. Finally notice that
\begin{equation}
\ri C=\ri E\subseteq E\subseteq \closu{E}=\closu{C}=\closu{(\ri C)}.
\end{equation}
That is $E\approx \ri C$ and hence $E$ is nearly convex
by Fact~\ref{l:ruger}\ref{l:rugerii+}.
\end{proof}

To study the relationship among core, interior and relative interior of a nearly convex set, we need two facts.
\begin{fact}\label{f:c:core:int}
Let $C$ be a convex set in $\RR^n$. Then $\intr C=\core C$.
Moreover, if $\intr C\neq \fady$ then $\intr C=\ri C$.
\end{fact}
\begin{proof}For the first part, see \cite[Remark 2.73]{Bonnans} or \cite[Proposition 6.12]{BC2011}. The second part is clear from \cite[pg 44]{Rock70}.
\end{proof}

\begin{fact}\emph{\cite[Proposition~2.20]{bmw-nenc}}\label{f:nc:int:cl}
Let $E$ be a nearly convex subset of $\RR^n$. Then
$\intr E=\intr (\closu{E})$.
\end{fact}

\begin{theorem}\label{nc:core:int}
Let $E$ be a nonempty nearly convex subset in $\RR^n$.
Then the following hold:
\begin{enumerate}
\item\label{nc:core:int:1}
$\core E=\intr E$.
\item\label{nc:core:int:2}
If $\intr E\neq \fady$ then $\intr E=\ri E$.
\item\label{nc:core:int:3}
$\aff (\ri E)=\aff E=\aff (\closu{E})$.
\end{enumerate}
\end{theorem}
\begin{proof}
\ref{nc:core:int:1}: Since $E$ is nonempty and nearly convex,
$\closu{E}$ is nonempty and convex by Fact~\ref{l:nearconset}\ref{o:c}.
Using Fact~\ref{f:nc:int:cl}
and Fact~\ref{f:c:core:int}
applied to the convex set $\closu{E}$ we have
\begin{equation}
\intr(\closu{E})=\intr E\subseteq \core E\subseteq \core (\closu{E})=\intr(\closu{E}).
\end{equation}
Hence $\intr E= \core E$, as claimed.

\ref{nc:core:int:2}: Notice that
Fact~\ref{f:nc:int:cl} gives
$\fady\neq \intr E= \intr (\closu{E})$.
Moreover since $E$ is nearly convex, Fact~\ref{l:nearconset}
implies that $\ri E=\ri (\closu{E})$.
Applying Fact~\ref{f:c:core:int} to
the convex set $\closu{E}$ gives
\begin{equation}
\intr E=\intr (\closu{E})=\ri  (\closu{E})=\ri E.
\end{equation}

\ref{nc:core:int:3}: %Notice that
%\begin{equation}\label{eq:aff:1}
%\aff \ri E\subseteq \aff E\subseteq \aff \closu{E}.
%\end{equation}
It follows from Fact~\ref{l:nearconset} that
\begin{equation}\label{eq:rec:n}
\ri E=\ri (\ri E)\;\;\; \text{and}\;\;\; \closu{(\ri E)}=\closu{E}.
\end{equation}
Using \eqref{eq:rec:n} and Fact~\ref{f:ri}\ref{f:ri:ii} applied to the convex set $\ri E$
we have
\begin{equation}\label{eq:aff:2}
\aff (\ri E)=\aff [\closu{(\ri E)}]=\aff (\closu{ E}).
\end{equation}
%Combining \eqref{eq:aff:1} and \eqref{eq:aff:2}
% and using the squeeze theorem we conclude that
% \eqref{eq:aff:3} holds true.}
\end{proof}

Under mild assumptions, a nearly convex set is in fact
convex as our next result shows.

\begin{definition} We say that a set $E\subseteq X$ is relatively
strictly convex if $]x,y[ \subseteq \ri E$ whenever $x,y\in E$.
\end{definition}

\begin{proposition}Let $E\subseteq \RR^n$ be nearly convex. Then the following hold:
\begin{enumerate}
\item\label{i:strict}  If $E$ is relatively strictly convex,
then $E$ is convex.
\item\label{i:open} If $E$ is  open, then $E$ is convex.
\item\label{i:closed} If $E$ is closed, then $E$ is convex.
\item\label{i:boundary} If  for every $x,y\in E\setminus\ri E$,
we have $[x,y]\subseteq E$, then $E$ is convex.
\end{enumerate}
\end{proposition}
\begin{proof}
\ref{i:strict}: Let $x,y\in E$. As $E$ is relatively strictly convex, $]x,y[\subseteq \ri E\subseteq E$, so
$[x,y]\subseteq E$. Hence $E$ is convex.

\ref{i:open}: For a nearly convex set $E$, we have
$\ri E=\ri C$ where $C$ is a convex set. When $E$ is open and $\inte E\neq \varnothing$, by
Theorem~\ref{nc:core:int}\ref{nc:core:int:2} we have
$E=
\ri E=\ri C=\inte C$ is convex.
Hence $E$ is convex.

\ref{i:closed}: Since there exists a convex set $C$ such that
$C\subseteq E\subset\closu{C}$ and $E$ closed, we have $E=\closu{C}$, so $E$ is convex.

\ref{i:boundary} Let $x,y\in E$. If one of $x,y$ is in $\ri E$, then $[x,y]\subseteq E$;
if both $x,y\in E\setminus \ri E$, then the assumption guarantees $[x,y]\subseteq E$.
Hence $E$ is convex.
\end{proof}

\section{Calculus and relative interiors of nearly convex sets}\label{s:calculus:interior}
In this section we extend the calculus for convex sets provided in \cite[Section 6 and Section 8]{Rock70} to nearly convex sets.
More precisely, we study the properties of images and pre-images of nearly convex sets under linear transforms.
One distinguished feature is that when two nearly convex sets are nearly equal,
their linear images or
linear inverse images are also nearly equal.
We start with

\begin{proposition}\label{p:productset}
\begin{enumerate}
\item\label{i:scalar}
Let  $E\subseteq\RR^n$ be nearly convex, and $\lambda\in \RR$. Then
$\reli (\lambda E) = \lambda \reli E.$
\item\label{i:product}
If $(\forall i\in I)\ E_i\subseteq \RR^n$ be nearly convex, then
$$\reli(E_1 \times\cdots\times E_m) = \reli E_1 \times\cdots\times \reli E_m,$$
and
$\closu{(E_1 \times\cdots\times E_m)} =\closu{(\reli E_1)} \times\cdots\times \closu{(\reli E_m)}.$
\end{enumerate}
\end{proposition}
\begin{proof}
\ref{i:scalar} As $E$ is nearly convex set, there exists a convex set $C\subseteq\RR^n$ such that
$\reli E=\reli C$. Then
$\reli(\lambda E)=\reli(\lambda C)=\lambda \reli C=\lambda \reli E$ by Fact~\ref{f:ri}\ref{f:ri:scal}.

\ref{i:product} By Fact~\ref{l:nearconset},
we have
\begin{align*}
\reli(E_1 \times\cdots\times E_m) & =\reli[\closu({E_{1}\times\cdots\times E_{m}})]=
\reli(\closu{E_{1}}\times\cdots\times\closu{E_{m}})\\
& =
\reli(\closu{E_{1}})\times\cdots\times \reli(\closu{E_{m}})=\reli E_{1}\times \cdots\times \reli E_{m}.
\end{align*}
Also by Fact~\ref{l:nearconset},
$\closu(E_1 \times\cdots\times E_m)=\closu{E_{1}}\times\cdots\times \closu{E_{m}}
=\closu{(\reli E_{1})}\times\cdots\times \closu{(\reli E_{m})}.$
\end{proof}

%\hl{Added June 26: the following is an alternative proof of the previous theorem}
\begin{theorem}\label{thm:lin:A}
Let $E$ be a nearly convex set in $\RR^n$ and
let $A: \RR^n \to \RR^m$ be a linear transformation.
Then the following hold:
\begin{enumerate}
\item\label{thm:lin:A:1}
$A(E)$ is nearly convex. %\hl{this result is NOT in the paper by Radu}
\item\label{thm:lin:A:2}
$AE\approx A(\ri E)\approx A(\closu{E})$.%\hl{this result is NOT in the paper by Radu}
%\item\label{thm:lin:A:3}
%$\reli (AC) = A \reli C.$
\item \label{201402264}
$\reli (AE) = A \reli E.$
\end{enumerate}
\end{theorem}
\begin{proof}
\ref{thm:lin:A:1}:
Since $E$ is nearly convex we have $\ri E$ is convex. It follows from Fact~\ref{l:nearconset}
that $\closu{E}=\closu{\ri E}$.
Moreover Fact~\ref{f:convlinear} applied to the convex set $\ri E$ implies that
$A(\closu{(\ri E)})\subseteq \closu{A(\ri E)}$.
Therefore,
\begin{equation}\label{AC:nc}
A(\ri E)\subseteq A(E)\subseteq A(\closu{E})= A(\closu{(\ri E)})\subseteq \closu{A(\ri E)}.
\end{equation}
Since $A$ is linear and $\ri E$ is convex, we conclude that
$A(\ri E)$ is convex, hence by \eqref{AC:nc} $A(E)$ is nearly convex.

\ref{thm:lin:A:2}: It follows from Fact~\ref{l:nearconset}, \eqref{AC:nc} and the fact that $A(\ri E)$ is convex
that
\begin{equation}\label{thm:lin:A:eq}
AE\approx A(\ri E).
\end{equation}
To show $ A(\ri E)\approx A(\closu{E})$, applying \eqref{thm:lin:A:eq}
to the convex set $\closu{E} $ we obtain
$A(\closu{E})\approx A(\ri (\closu{E}))$.
Since $E$ is nearly convex, we have $\ri E= \ri (\closu{E}) $ by Fact~\ref{l:nearconset},
hence $A(\closu{E})\approx A(\ri {E})$, and \ref{thm:lin:A:2} holds.

\ref{201402264}: By \ref{thm:lin:A:2} and Fact~\ref{f:convlinear}, we have
$$\ri (AE)=\ri[A(\closu E)]=A(\ri (\closu{E}))=A(\ri E).$$
\end{proof}

\begin{remark} Theorem~\ref{thm:lin:A}\ref{thm:lin:A:1}\&\ref{201402264} was proved in \cite[Lemmas 2.3, 2.4]{Bot2008},
our proof is different from theirs.
\end{remark}

\begin{corollary}
Let $(\forall i\in I)\ E_i$ be nearly convex sets in $\RR^n$. Then
$$\reli(E_1+\cdots+E_{m})=\reli E_1 +\cdots+ \reli E_m.$$
\end{corollary}
\begin{proof}
Apply Theorem~\ref{thm:lin:A}\ref{201402264} and Proposition~\ref{p:productset}\ref{i:product}
 with $A:(\bx_{1},\ldots,\bx_{m}) \mapsto \sum_{i\in I}\bx_{i}$
where $\bx_{i}\in\RR^n$, and
$E:=E_{1}\times \cdots\times E_{m}$.
\end{proof}

\begin{theorem}\label{201409121}
Let $A:\RR^n \to \RR^m$ be a linear transformation
and let $E \subseteq \RR^m$ be a nearly convex set
such that $A^{-1}(\reli E) \neq \varnothing$. Then
\begin{enumerate}
\item\label{i:inversesetn} $A^{-1}E$ is nearly convex,
\item \label{201409121:1}$\reli(A^{-1}E) = A^{-1}(\reli E),$
\item \label{201409121:2} $\closu{[A^{-1}(E)]}=A^{-1}(\closu{E}).$
\end{enumerate}
\end{theorem}
\begin{proof} As $E$ is nearly convex, there exists
a convex set $C$ such that $C \subseteq E \subseteq \closu{C}$
and $\reli E = \reli C$. The assumption $A^{-1}(\reli E)\neq \varnothing$
is equivalent to $A^{-1}(\reli C) \neq \varnothing$, so $A^{-1}(\reli \closu{C)} \neq \varnothing$
by
Fact~\ref{f:ri}\ref{f:ri:iii}.
Because $A^{-1}(\reli C) \neq \varnothing$, by Fact~\ref{f:convlinearinv},
we have $\closu {[A^{-1}(C)]} = A^{-1}(\closu{C})$. Then
$$A^{-1}(C) \subseteq A^{-1}(E) \subseteq A^{-1}(\closu{C})
\subseteq \closu{[A^{-1}(\closu{C})]} = \closu{[A^{-1}({C})]}$$
which gives
$A^{-1}(C) \subseteq A^{-1}(E) \subseteq \closu{[A^{-1}(C)]},$
so \ref{i:inversesetn} holds.
It also follows that $\reli(A^{-1}(E)) = \reli (A^{-1}(C))$.
By Fact~\ref{f:convlinearinv},
$$\reli(A^{-1}(C)) = A^{-1}(\reli C) = A^{-1}(\reli E).$$
Therefore \ref{201409121:1} holds.
\ref{201409121:2} follows from
$$\closu{[A^{-1}(E)]} = \closu{[A^{-1}(C)]}=A^{-1}(\closu{C}) = A^{-1}(\closu{E}).$$
\end{proof}

\begin{remark}
Theorem~\ref{201409121} was proven in \cite[Theorem~2.2, Corollary 2.1]{Bot2008}.
However our proof is different.
\end{remark}

\begin{theorem}\label{thm:lin:Ainv}
Let $A:\RR^n \to \RR^m$ be a linear transformation
and let $E \subseteq \RR^m$ be a nearly convex set
such that $A^{-1}(\reli E) \neq \varnothing$. Then
%the following hold:
%\begin{enumerate}
%\item\label{thm:lin:Ainv:2}
$A^{-1}(E)\approx A^{-1}(\ri E)\approx A^{-1}(\closu{E})$.
\end{theorem}
\begin{proof}
First notice that since $E$ is nearly convex we have
$\ri E$ and
$\closu{E}$ are convex and $\ri E=\ri{(\closu{E})}$,
hence $A^{-1}(\reli (\closu{E}))=A^{-1}(\reli E)  \neq \varnothing$.
%\ref{thm:lin:Ainv:2}:
Using Fact~\ref{l:nearconset} we have
\begin{equation}\label{sub:0}
\ri E=\ri (\ri E).
\end{equation}
Applying Fact~\ref{f:convlinearinv} to the convex sets $\ri E$ and $\closu{E}$ we have
\begin{align}
A^{-1}(\ri E)&=A^{-1}(\ri (\ri E))=\ri A^{-1}(\ri E)\label{sub:1}\\
A^{-1}(\ri E)&=A^{-1}( \ri (\closu{E}))=\ri A^{-1}(\closu{E}).\label{sub:2}
\end{align}
 Since $\ri E$ and
$\closu{E}$ are convex we have
\begin{equation} \label{sub:2+}
A^{-1}(\ri E) \quad\text{and}\quad
A^{-1}(\closu{E}) \quad\text{are convex}.
\end{equation}
Moreover, it follows from \eqref{sub:1} and \eqref{sub:2} that
$\ri A^{-1}(\ri E)=\ri A^{-1}(\closu{E})$.
Therefore, using Fact~\ref{ri:part} we conclude that
\begin{equation}\label{sub:3}
A^{-1}(\ri E)
\approx A^{-1}(\closu{E}).
\end{equation}
Notice that $A^{-1}(\ri E) \subseteq A^{-1}(E)
\subseteq  A^{-1}(\closu{E}).$ Therefore using \eqref{sub:3} and Fact~\ref{sq:thm}
we conclude that  $A^{-1}(E)\approx A^{-1}(\ri E)\approx A^{-1}(\closu{E})$.
%\ref{thm:lin:Ainv:1}: Combine \eqref{sub:2+}, \ref{thm:lin:Ainv:2} and Fact~\ref{l:ruger}~\ref{l:rugerii+}.
\end{proof}

The next result generalizes Rockafellar's Fact~\ref{f:riintersec} to nearly convex sets. Our proof
is different from the one given in \cite[Theorem 2.1]{Bot2008}.

\begin{corollary}\label{ncriintersecm}
Let $E_i$ be nearly convex sets in $\RR^n$ for all $i\in I$
such that $\bigcap\limits_{i=1}^m \reli E_i \neq \varnothing$.
Then the following hold:
\begin{enumerate}
\item\label{it:ncriintersecm:1}
$\bigcap\limits_{i=1}^m E_i$ is nearly convex.
\item\label{it:ncriintersecm:2}
$\bigcap\limits_{i=1}^m \reli E_i = \reli \bigcap\limits_{i=1}^m E_i.$
\item \label{it:ncriintersecm:3}
$\closu{\bigcap\limits_{i=1}^m E_i} = \bigcap\limits_{i=1}^m \closu{E_i}.$
\end{enumerate}
\end{corollary}
\begin{proof}
Define $A:\RR^n\rightarrow \RR^n\times\cdots\times \RR^n$ by
$A\bx:=(\bx,\ldots,\bx)$ where $\bx\in\RR^n$. The set
$E:=E_{1}\times\cdots\times E_{n}$ is nearly convex.
The results follow by combining Theorems~\ref{201409121}, \ref{thm:lin:Ainv}.
\end{proof}

%\hl{The following corollary is added June 25. This result is NOT in the paper by Radu}
\begin{cor}
Let $E_1$ and $E_2$ be nearly convex sets in $\RR^n$ such that
$E_1\approx E_2$ and
let $A: \RR^n \to \RR^m$ be a linear transformation.
The following hold.
\begin{enumerate}
\item\label{i:imagel} $AE_1\approx AE_2$,
\item \label{i:inverse:image}
If $A^{-1}(\reli E_{1})\neq\varnothing$, then $A^{-1}E_{1}\approx A^{-1}E_{2}$.
\end{enumerate}
\end{cor}
\begin{proof}
Indeed, since $E_1\approx E_2$ we have $\ri E_1=\ri E_2$.
It follows from Theorem~\ref{thm:lin:A} that
\begin{equation}
AE_1\approx A(\ri E_1)=A(\ri E_2)\approx AE_2.
\end{equation}
This gives \ref{i:imagel}. For \ref{i:inverse:image}, apply Theorem~\ref{thm:lin:Ainv}.
%\ref{thm:lin:Ainv:2}.
\end{proof}

\section{Recession cones of nearly convex sets}\label{s:recession:cone}
%\hl{ADDED: August 31}
In this section we extend several of the results for the calculus of recession cones in \cite[Chapter~8]{Rock70} to nearly convex sets. Intuitively, it would seem that most recession cone results on
convex sets should hold for nearly convex sets. Unfortunately, this is not the case.
On the positive side, we establish that $\rec E$ of a nearly convex set $E$
is nearly convex provided that
$\spn(E-E)=\rec(\closu E)-\rec(\closu E).$

\begin{fact}\emph{\cite[Theorem~8.1]{Rock70}}
Let $C$ be a nonempty convex subset of $\RR^n$.
Then $\rec C$ is a convex cone and $0\in \rec C$.
Moreover,
\begin{equation}\label{eq:rec:conv}
\rec C=\stb{y \in \RR^n \suchthat y+C\subseteq C}.
\end{equation}
\end{fact}
When $E$ is nonempty nearly convex subset of $\RR^n$, the characterization
\eqref{eq:rec:conv} may fail to be equivalent to \eqref{eq:rec:d}
as illustrated in the following example. Let $\NN$ denote the set of natural numbers.

\begin{example}\label{ex:rec:nc:i}
Suppose that $E:=\stb{(x,y)~|~y\ge x^2,x\ge 0}\setminus \big(\stb{0}\times \NN\big)\subseteq \RR^2$.
Notice that the set $\tilde{E}=\stb{(x,y)~|~y\ge x^2,x\ge 0}$ is convex and
$E\approx \tilde{E}$ and therefore $E$
is nearly convex by Fact~\ref{l:ruger}\ref{l:rugerii+}.
Clearly $(\forall y\in \stb{0}\times \NN)$
$y+E\subseteq E$, however $y\not\in\rec{E}=\stb{0}$,
since $(\forall \lam \in R_+\setminus \NN)$ $\lam y+E\not\subseteq E$.
\end{example}

\begin{lem}\label{NC:recriC:recclC}
Let $E$ be a nonempty nearly convex subset in $\RR^n$.
Then $\rec (\ri E)=\rec (\closu{E})$, in particular, a closed convex cone.
\end{lem}
\begin{proof}
Since $E$ is nearly convex, it follows that $\ri E$ and $\closu{E} $ are convex sets. Fact~\ref{l:nearconset} gives
$\ri(\ri E)=\ri E$ and $\closu{(\ri E)}=\closu{E}$. Applying Fact~\ref{f:recC:II} to the convex set $\ri E$
 we conclude that
 \begin{equation}
 \rec (\ri E)=\rec[\ri (\ri E)]=\rec[\closu{(\ri E)}]=\rec(\closu{ E}).
 \end{equation}
As $\closu{E}$ is closed convex, Fact~\ref{f:recC} completes the proof.
\end{proof}

\begin{proposition}\label{p:product:rec}
 Given $(\forall i\in I)$ $E_{i}\subseteq\RR^n$. Then the following hold:
\begin{enumerate}
\item\label{i:rec:prod1}
 $\rec (E_{1}\times \cdots\times E_{m})=\rec E_{1}\times\cdots\times \rec E_{m}$.
\item\label{i:rec:prod2}
 If, in addition, each $E_{i}$ is nearly convex, then
$$\rec [\closu(E_{1}\times \cdots\times E_{m})]=\rec (\closu{E_{1}})\times\cdots\times \rec (\closu{E_{m}})
=\rec (\ri {E_{1}})\times\cdots\times \rec (\ri{E_{m}}).$$
\end{enumerate}
\end{proposition}
\begin{proof}
\ref{i:rec:prod1}: This follows from the definition recession cone \eqref{eq:rec:d}.

\ref{i:rec:prod2}: By \ref{i:rec:prod1} and Lemma~\ref{NC:recriC:recclC}, we have
\begin{align*}
\rec [\closu(E_{1}\times \cdots\times E_{m})]& =\rec (\closu E_{1}\times \cdots\times \closu E_{m})\\
& =\rec (\closu{E_{1}})\times\cdots\times \rec (\closu{E_{m}})\\
& =\rec (\ri {E_{1}})\times\cdots\times \rec (\ri{E_{m}}).
\end{align*}

\end{proof}

The following result is folklore, and we omit its proof.
\begin{fact}\label{f:folklore}
Let $S\subseteq\RR^n$ be nonempty, and $K\subseteq \RR^n$ be a convex cone.
Then the following hold:
\begin{enumerate}
\item\label{C:cl:bd}
%\hl{Again I know this result is well-known but can't find a reference}
The set $S$  is bounded
if and only if $\closu{S}$ is bounded.
\item \label{A:D} For every $x\in S$,
 \begin{equation}
 {\aff}(S)=\stb{x}+{\spn(S-S)}.
 \end{equation}
 In addition, if $0\in S$, then
 $
 \spn S= \aff S=\spn(S-S).
 $
 \item \label{i:span:same} $\spn(S-S)=\spn(\closu S-\closu S)$.
 \item \label{i:convex:cone}
 $K-K=\spn K$.
 \end{enumerate}
 \end{fact}

 %\begin{proof}
% Let $x\in S$ and set $E:=\stb{x}+\spn(S-S)$.
% Let $w=x+\sum_{i=1}^{m}\lam _i(u_i-v_i)\in E$,
% with $u_i\in S $, $v_i\in S$ for all $i\in I$.
% Set $w=\sum_{i=0}^{2m}\alpha_is_i$, where $\alpha_0=1$,
% $\alpha_{2i-1}=\lam_i$, $\alpha_{2i}=-\lam_i$,
% $s_0=x$, $s_{2i-1}=u_i$, $s_{2i}=v_i$ for all $i\in I$. That is, $\sum_{i=0}^{2m}\alpha_i=1$ and $s_i\in S$ for all
% $i\in I$, and therefore $E\subseteq \aff(S)$.
%
% Conversely, let $w=\sum_{i=1}^{m}\lam_i s_i\in \aff(S)$,
% where $\sum_{i=1}^{m}\lam_i =1$, $s_i\in S$, for all $i\in I$.
% Then $w=x+\sum_{i=1}^{m}\lam_i (s_i-x)\in E$. That is, $\aff(S)=E$,
% as claimed.
%% Hence, ${\aff}(S)={\stb{x}+\spn(S-S)}
%% =\stb{x}+{\spn(S-S)}$,
%% which completes the proof.
% \end{proof}
\begin{fact}\emph{\cite[Theorem~8.4]{Rock70}}\label{f:cl:i}
Let $C$ be a nonempty closed convex subset in $\RR^n$.
Then $C$ is bounded if and only if $\rec C=\stb {0}$.
\end{fact}

\begin{prop}
Let $E$ be a nonempty nearly convex subset in $\RR^n$.
Then $E$ is bounded if and only if
$\rec (\closu{E})=\stb{0}$.
\end{prop}
\begin{proof}
%Suppose that $E$ is bounded.
Since $E$ is nearly convex,
%it follows from Fact~\ref{l:nearconset}\ref{o:c}
%and Lemma~\ref{C:cl:bd}
$\closu{E}$ is a nonempty closed convex.
Using Fact\ref{f:folklore}\ref{C:cl:bd} and Fact~\ref{f:cl:i} applied to $\closu{E}$ we conclude that
\begin{equation}
E \text{\;\; is bounded\;\;}\iff \closu{E} \text{\;\; is bounded\;\;}\iff \rec (\closu{E})
=\stb{0}.
\end{equation}
\end{proof}

The following example shows that $\rec(\closu{E})=\{0\}$ cannot be replaced by $\rec E=\{0\}$.
\begin{example} The almost convex set
$$E:=\menge{(x_{1},x_{2})}{-1\leq x_{1}\leq 1, x_{2}\in\RR}\setminus
\menge{(x_{1},x_{2})}{x_{1}=\pm 1, -1<x_{2}<1}$$
has $\rec E=\{(0,0)\}$, but $E$ is unbounded.
\end{example}

%\begin{lem}\label{aux:1}
% Let $S$ be a nonempty subset of $\RR^n$. Then
% \begin{equation}
% \spn\rec S= \aff\rec S=\spn(\rec S-\rec S)\subseteq\spn (S-S).
% \end{equation}
%\end{lem}
%\begin{proof}
%Since $0\in \rec S$ we have  $\spn\rec S= \aff\rec S= \spn(\rec S-\rec S)$.
%Let $y\in \spn (\rec S-\rec S)$. Then $(\exists (\lam_i)_{i\in I}\subseteq \RR)$
%$(\exists (r_i)_{i\in I}\subseteq \rec S)$ and $(\exists (s_i)_{i\in I}\subseteq\rec S)$
%such that $y=\sum_{i\in I} \lam_i (r_i-s_i)$. Now let $x\in S$ and notice that
%$(\forall i\in I) x+r_i\in S$ and $(\forall i\in I) x+s_i\in S$.
%Therefore $y=\sum_{i\in I} \lam_i (x+r_i-(x+s_i))\in \spn (S-S)$,
%which completes the proof.
%\end{proof}

\begin{lem}\label{lem:rec:reccl}
%\hl{I believe this result is known but I could not find a reference}
Let  $S$ be a nonempty subset of $\RR^n$. Then $\rec S \subseteq\rec(\closu{S})$.
\end{lem}
\begin{proof}
Let $y\in \rec S$ and let $s\in \closu{S}$.
Then $(\exists (s_n)_{n\in \NN})\subseteq S$ such that
$s_n\to s$. Since $(\forall n\in \NN)$ $s_n\in S$,
it follows from that definition of $\rec S$ that
 $(\forall n\in \NN)$ $(\forall \lambda\geq 0)$ $\lambda y+s_n\in S$, hence
 $\lambda y+s_n\to \lambda y+s\in \closu{S}$.
 That is $y\in \rec (\closu{S})$, which completes
 the proof.
\end{proof}

We are now ready for the main result of this section.
\begin{prop}\label{P:rec:NC}
Let $E$ be a nonempty nearly convex subset of $\RR^n$.
Suppose that
\begin{equation}\label{e:span:rece}
\spn[\rec (\closu{E})]= \spn (\closu{E}-\closu{E}),
\end{equation}
equivalently,
\begin{equation}\label{e:span:rece1}
\rec(\closu E)-\rec(\closu E)=\spn (E-E).
\end{equation}
Then the following hold:
\begin{enumerate}
\item\label{P:rec:NC:1}
$\ri [\rec(\closu{E})]\subseteq \rec E$.
\item\label{P:rec:NC:2}
$\rec E$ is nearly convex.
\item\label{P:rec:NC:3}
$\rec E\approx \ri [\rec (\closu{E})]$.
\item\label{P:rec:NC:4}
$\rec (\ri E) \approx \rec E\approx\rec ( \closu{E})$.
\end{enumerate}
\end{prop}
\begin{proof}
Observe that \eqref{e:span:rece} and \eqref{e:span:rece1} are equivalent because of
Fact~\ref{f:folklore}\ref{i:span:same}\&\ref{i:convex:cone}.

\ref{P:rec:NC:1}:
Let $y\in \ri[\rec(\closu{E})]\subseteq \rec(\closu{E})$.
Then $(\exists \eps >0)$
such that
\begin{equation}\label{e:y:inreli}
B(y,\eps)\cap\aff[\rec (\closu{E})]\subseteq \rec (\closu{E}).
\end{equation}
Therefore $(\forall x \in \closu{E})$,
 $x+y\in \closu{E}$, and
\begin{equation}\label{e:asmp:i}
 x+\big(B(y,\eps)\cap\aff[\rec (\closu{E})]\big)\subseteq  \closu{E}.
\end{equation}
Using Fact~\ref{f:folklore}\ref{A:D}, \eqref{e:span:rece} and \eqref{e:asmp:i}
we have
\begin{align}\label{e:asmp:ii}
B(x+y,\eps)\cap \aff(\overline{E})
&= B(x+y,\eps)\cap(x+\spn ( \overline{E}-\overline{E}))\nonumber\\
&= x+(B(y,\eps)\cap(\spn ( \overline{E}-\overline{E})))\nonumber\\
&= x+\big(B(y,\eps)\cap \aff [\rec (\closu{E})]\big)\nonumber\\
&\subseteq  \closu{E}.
\end{align}
That is, ${x+y}\in \ri(\closu{E}) $.
Since $E$ is nearly convex we have
$\ri (\closu{E})=\ri E$, hence
${x+y}\in \ri {E} \subseteq E$.
%$(\exists w\in \spn(\rec\overline{C}-\rec\overline{C} ))$
%with $\norm{w}< \rho$ such that
%\begin{equation}\label{eq:rec:1}
%y+w\in \rec \overline{C}.
%\end{equation}
%It follows from Lemma~\ref{aux:1} and \eqref{eq:rec:1}
%that $w\in \spn (\overline{C}-\overline{C})$ with $\norm{w}<\rho$
%and $(\forall \overline{c}\in \overline{C} )$
%we have $y+w+c\in \overline{C}$.
%Since $C$ is nearly convex we have
%$\ri \overline{C}=\ri C$.
In particular, it follows from \eqref{e:asmp:ii} that
$(\forall z \in {E})  \ z+y\in E$.
Note that $\rec(\closu{E})$ is a cone, and $\aff[\rec (\closu{E})]$ is a subspace.
For every $\lambda>0$, $\lambda y\in \ri[\rec(\closu{E})]$,
multiplying \eqref{e:y:inreli} by $\lambda>0$ gives
$$B(\lambda y,\lambda\eps)\cap\aff[\rec (\closu{E})]\subseteq \rec (\closu{E}).$$
The above arguments show that $(\forall z \in {E})  \ z+\lambda y\in E$.
%\ri \overline{E}=\ri E\subseteq E$,
Consequently $y\in \rec E$, as claimed.

\ref{P:rec:NC:2}\&\ref{P:rec:NC:3}:
Since $E$
is nearly convex we have $\closu{E}$
 is a nonempty closed convex set.
Therefore by Fact~\ref{f:recC} and
Fact~\ref{f:ri}\ref{f:ri:i--}\&\ref{f:ri:i}
applied to the convex set $\rec (\closu{E})$
we have
$\rec (\closu{E})$ is a nonempty closed convex cone,
$\ri [\rec (\closu{E})]$ is a nonempty convex set
and $$\rec (\closu{E})=\closu{ [\rec (\closu{E})]}
=\closu{\big(\ri [\rec (\closu{E})]\big)}.$$
It follows from Lemma~\ref{lem:rec:reccl} that
 $\rec E\subseteq \rec(\closu{E})$.
Therefore using \ref{P:rec:NC:1} we have
\begin{equation}\label{eq:inc:im}
\ri [\rec (\closu{E})]\subseteq \rec E\subseteq \rec(\closu{E})
=\closu {[\rec (\closu{E})]}=\closu {\big(\ri [\rec (\closu{E})]\big)}.
\end{equation}
Using Fact~\ref{l:nearconset} we conclude that $\rec E$
is nearly convex and
$\rec E\approx \ri [\rec (\closu{E})]$, as claimed.
\ref{P:rec:NC:4}:
It follows from \eqref{eq:inc:im} that
\begin{equation}\label{eq:cl:part}
\closu{(\rec E)}=\closu{[\rec(\closu{E})]}.
%=\closu {\big(\ri [\rec (\closu{E})]\big)}.
\end{equation}
%Using Proposition~\ref{P:rec:NC}\ref{P:rec:NC:2} and
%Fact~\ref{f:recC} applied to
%the convex set $\closu{E}$
Since $\rec E$ is nearly convex by \ref{P:rec:NC:2}, and
$\rec(\closu{E})$ is convex by Fact~\ref{f:recC}, it follows from \eqref{eq:cl:part}
and Fact~\ref{ri:part} that $\rec E\approx \rec(\closu{E})$.
Now combine with Lemma~\ref{NC:recriC:recclC}.
\end{proof}

%\begin{example}\label{recnotconvex}
%Suppose that $E=\stb{(x,y,z)~|~x>0,y>0,z\ge 0}\cup R_+\cdot(1,1,0)\subseteq \RR^3$.
%Then $E$ is not convex, however,
%$E\approx \stb{(x,y,z)~|~x\ge 0,y\ge 0,z\ge 0}=\overline{E}$,
%and therefore $E$ is nearly convex by Fact~\ref{l:ruger}~\ref{l:rugerii+}.
%Moreover, $\rec E=E\approx \overline{E}=\rec \overline{E}=\rec \ri E$.
%\end{example}

The following example shows that in Theorem~\ref{P:rec:NC}, condition \eqref{e:span:rece} cannot be removed.
\begin{example}
Suppose that $E$ is as defined in Example~\ref{ex:rec:nc:i}.
Then $\stb{0}=\rec E\not\approx \rec \overline{E}=R_+\cdot(0,1)$.
Note that \eqref{e:span:rece} fails because
$\spn(E-E)=\RR^2\neq \{0\}\times \RR=\spn[\rec(\closu{E})].$
\end{example}

Unfortunately, in general we do not know whether the recession cone of a nearly convex set is nearly convex.
We leave this as an open question.
%, if $C$ is nearly convex.
 %does this imply that $\rec C$ is nearly convex?
%\end{remark}
%\begin{cor}\label{n:eq:rec}
%Let $C$ be a nonempty nearly convex subset in $\RR^n$.
%Suppose that $ \spn\rec S= \spn (S-S)$.
%Then
%\begin{equation}
%      \rec \ri C\approx \rec C\approx\rec \overline{C}.
%\end{equation}
%\end{cor}
%\begin{proof}
%It follows from \eqref{eq:inc:im} that
%\begin{equation}\label{eq:cl:part}
%\overline{\rec C}=\overline {\ri \rec \overline{C}}.
%\end{equation}
%Using Proposition~\ref{P:rec:NC}~\ref{P:rec:NC:2} and
%Fact~\ref{f:recC} applied to the convex set $\overline{C}$
%we have that $\rec C$ is nearly convex and
%$\rec\overline{C}$ is convex. It follows from \eqref{eq:cl:part}
%and Fact~\ref{ri:part} that $\rec C\approx \rec\overline{C}$.
%Now combine with Lemma~\ref{NC:recriC:recclC}.
%\end{proof}
\section{Applications}\label{s:application:max:further}
In this section, we apply results in Section~\ref{s:recession:cone} and Section~\ref{s:calculus:interior}
to study the maximality of a sum of maximally monotone operators and
the closedness of linear images of nearly convex sets.

\subsection{Maximality of a sum of maximally monotone operators}

\begin{fact}{\rm(See, e.g., \cite[Corollary~12.44]{Rock98}.)}
\label{maxm_rock}
Let $A:\RR^n \rra \RR^n$ and $B:\RR^n\rra \RR^n$
be maximally monotone such that
$\ri \dom A\cap\ri \dom B\neq \fady$.
Then $A+B$ is maximally monotone.
\end{fact}

One can apply Fact~\ref{maxm_rock} and Corollary~\ref{ncriintersecm} to show the following well-known result.

\begin{theorem}\label{sum:n:mono}
Let $\stb{A_i}_{i\in I}$ be finite family of maximally monotone operators from $\RR^n \To \RR^n$ such that $\cap_{i=1}^{m}\ri (\dom A_i)\neq \fady$. Then $A_1+\cdots+A_m$ is maximally monotone.
\end{theorem}
\begin{proof}
We proceed via induction. When $m=2$, the proof follows from Fact~\ref{maxm_rock}. Next, assume that for
$m\in \NN $ with $m\geq 2$
we have $\cap_{i=1}^{m}\ri(\dom A_i)\neq \fady$ implies that
$A_1+\cdots+A_m$ is maximally monotone.
Now suppose that $\cap_{i=1}^{m+1}\ri (\dom A_i)\neq \fady$.
By Fact~\ref{D-R-nc}, for all $i\in \stb{1,\ldots,m+1}$, $\dom A_i$ is nearly convex set.
Moreover, by Corollary~\ref{ncriintersecm}\ref{it:ncriintersecm:2},
\begin{align*}
 \cap_{i=1}^{m+1}\ri \dom A_i&=\cap_{i=1}^{m}\ri \dom A_i\cap \ri \dom A_{m+1}\\
&=\ri\big( \cap_{i=1}^{m} \dom A_i\big)\cap \ri \dom A_{m+1}\\
&=\ri \dom (A_1+\cdots+A_m)\cap \ri \dom A_{m+1}.
\end{align*}
Therefore, $\cap_{i=1}^{m+1}\ri \dom A_i\neq \fady$ implies
that $\cap_{i=1}^{m}\ri \dom A_i\neq\fady$ and
that $\ri \dom (A_1+\cdots+A_m)\cap \ri \dom A_{m+1}\neq \fady$.
By the inductive hypothesis we have $A_1+\cdots+A_m$ is maximally monotone. The proof then follows from
applying Fact~\ref{maxm_rock} to
the maximally monotone operators $A_1+\cdots+A_m$ and $A_{m+1}$.
\end{proof}

\subsection{Further closedness results}
\begin{theorem}\label{nc:recC:CL:I}
Let $E$ be a nonempty nearly convex subset in $\RR^n$
and let $A$ be a linear transformation from $\RR^n$ to $\RR^m$.
Suppose that $(\forall z\in \rec(\closu{E})\setminus\stb{0})$
with $Az=0$ we have that $z$ belongs to the lineality space of
$\closu{E}$. Then
$$\closu{AE}=A(\closu{E}),$$
and
$$\rec A(\closu{E})=A(\rec(\closu{E})).$$
\end{theorem}
\begin{proof}
Since $E$ is nonempty and nearly convex,
it follows from Fact~\ref{l:nearconset} that
$\ri E$ is a nonempty convex subset
and $\closu{(\ri E)}=\closu{E}$.
Therefore by assumption on $E$ we have
$(\forall z\in \rec(\closu{(\ri E)})\setminus\stb{0}=\rec(\closu{E})\setminus\stb{0})$
with $Az=0$ we have that $z$ belongs to the lineality space of
$\closu{(\ri E)}=\closu{E}$.
Using Theorem~\ref{thm:lin:A}\ref{thm:lin:A:2} we have
\begin{equation}\label{eq:AC:AC}
\closu{AE}=\closu{[A(\ri E)]}.
\end{equation}
Using \eqref{eq:AC:AC}, Fact~\ref{f:recC:CL:I} applied to the nonempty convex set $\ri E$,
and Fact~\ref{l:nearconset},
we obtain
\begin{align}
\closu{AE}&=\closu{[A(\ri E)]}=A(\closu{(\ri E)})=A(\closu{ E}), \nonumber\\
\rec A(\closu{E})&=\rec A(\closu{(\ri E)})=A(\rec(\closu{(\ri E)}))=A(\rec(\closu{E})),
\end{align}
as claimed.
\end{proof}

As a consequence, we have:
\begin{corollary}\label{nc:recC:Cl:II}
Let $(E_i)_{i\in I}$ be a family of
nonempty nearly convex subsets in $\RR^n$
satisfying the following condition:
if $(\forall i\in I)( \exists z_i\in \rec(\closu E_i))$
and $\sum_{i\in I} z_i=0$ then $(\forall i \in I) z_i$ belongs to the lineality space
of $\closu{E_i}$. Then
$$\closu{(E_1+\cdots+E_m)}=\closu{E_1}+\cdots+\closu{E_m}=
\closu{(\ri E_{1})}+\cdots+\closu{(\ri E_{m})},$$
and
$$\rec[\closu{(E_1+\cdots+E_m)}]=\rec(\closu{E_1})+\cdots+\rec(\closu{E_m})=
\rec(\ri{E_1})+\cdots+\rec(\ri{E_m}).$$
\end{corollary}
\begin{proof}
Define a linear mapping $A:\RR^n\times \cdots\times \RR^n \rightarrow \RR^n$ by
$A(\bx_{1},\ldots, \bx_{m}):=\bx_{1}+\cdots+\bx_{m}$ where $\bx_{i}\in\RR^n$. The set
$E:=E_{1}\times \cdots\times E_{m}$ is nearly convex in $\RR^n\times \cdots\times \RR^n$.
It suffices to apply Theorem~\ref{nc:recC:CL:I}, Lemma~\ref{NC:recriC:recclC},
and Proposition~\ref{p:product:rec}.
\end{proof}

Corollary~\ref{nc:recC:Cl:II} generalizes Fact~\ref{f:recC:Cl:II} from convex sets to nearly
convex sets.

\section{Examples of nearly convex sets as ranges or domains of subdifferential operators}\label{s:range:domain}
%\hl{ADDED: August 25}
It is natural to ask: is every nearly convex set a domain or a range of the subdifferential of a lower semicontinuous
convex function, or a domain or a range of a maximally monotone operator? While we cannot answer the question,
we construct some interesting proper lower semicontinuous convex functions with prescribed domains or
ranges of subdifferentials. Our constructions rely on the sum rule of subdifferentials for a sum of convex functions,
while each convex function has a subdifferential domain with specific properties.

Recall that for a proper lower semicontinuous convex function $f: \RR^n \to \RX$, its
%The epigraph of $f$ is $\epi f = \{(x,y) \suchthat x\in \RR^n,~y\in \RR, \text{ and } f(x)\leq y\}$ and
subdifferential at $x \in \RR^n$ is
$$\partial f(x):=\{ u \in \RR^n \suchthat (\forall y \in \RR^n) \; \scal{y-x}{u}+f(x)\leq f(y)\}$$
when $f(x)<+\infty$; and $\partial f(x):=\varnothing$ when $f(x)=+\infty$.
If $f$ is continuous at $x$ (e.g., when $x\in\inte\dom f$), then $\partial f(x)\neq\varnothing$;
if $f$ is differentiable at $x$, then $\partial f(x)=\{\nabla f(x)\}$. Fact~\ref{f:subcharacter}
provides a convenient tool to compute
$\partial f$. A celebrated result due to Rockafellar states that the subdifferential mapping
$\partial f$ is a maximally monotone operator, see, e.g., \cite[Theorem 12.17]{Rock98}, \cite[Theorem 3.1.11]{Zalinescu}.
In \cite[page 218]{Rock70}, Rockafellar gave a proper lower semicontinuous convex function
whose subdifferential domain is not convex. (Rockafellar's function is Example~\ref{Rock:ex:gen} when $\alpha =1$.)
According to Fact~\ref{D-R-nc}, $\dom \partial f$ must be nearly convex.  The Fenchel conjugate $f^*$ of $f$ is defined by
$$(\forall\ x^*\in\RR^n)\ f^*(x^*):=\sup\menge{\scal{x^*}{x}-f(x)}{x\in\RR^n}.$$
The conjugate $f^*$ is a proper lower semicontinuous convex function as long as $f$ is, and
$\partial f^*=(\partial f)^{-1}$.
 The indicator function of a set $C\subseteq\RR^n$ is
$$\iota_C(x):=\begin{cases}0 & \text{if }x\in C;\\+\infty & \text{otherwise.}\end{cases}$$

A brief orientation about our main achievements in this section is as follows. We show that: Every open or closed convex
set
 in $\RR^n$
is a domain of a subdifferential mapping, so are their intersections under a constraint qualification;
Every nearly convex set in $\RR$ is
a domain of a subdifferential mapping; In $\RR^2$ every polyhedral sets with its edges removed but keeping its vertices
is a domain of a subdifferential mapping.

\subsection{Some general results}

A set $A\subseteq\RR^n$ is absorbing if for every $x\in\RR^{n}$ there exists $s_{x}$ such that
$x\in tA$ when $t>s_{x}$. Define the gauge function of $A$ by
$$\rho_{A}(x):=\inf\menge{t}{t>0, x\in tA}.$$
Then $\rho_{A}$ is finite-valued, nonnegative, and positive homogeneous.

\begin{theorem}\label{t:openset}
Let $C\subseteq \RR^n$ be open convex set. Then there exists a lower semicontinuous convex function
$g:\RR^n\rightarrow\RX $ such that
 such that $\ran \partial g =C$, equivalently, $\dom \partial g^*=C.$
\end{theorem}
\begin{proof}
Take $x_0\in C$, and let $A:=C-x_{0}$. Define the lower semicontinuous convex function
$$g(x):=\begin{cases}
\frac{1}{1-\rho_{A}(x)} & \text{ if $x\in A$,}\\
+\infty & \text{ otherwise}.
\end{cases}
$$
The convexity of $g$ follows from that $\rho_{A}$ is a finite-valued convex function
on $\RR^n$,
$A=\menge{x\in\RR^n}{\rho_{A}(x)<1}$ by \cite[Exercise 2.15]{Fabian}, and that
$t\rightarrow \frac{1}{1-t}$ is increasing and convex on $[0,1)$.
Since
$$\partial g(x)=\begin{cases}
\frac{\partial \rho_{A}(x)}{(1-\rho_{A}(x))^2} & \quad (\forall x\in A),\\
\varnothing & \text{otherwise,}
\end{cases}
$$
we have $\dom \partial g=A$, so $\ran \partial g^*=\ran (\partial g)^{-1}=A$.
Then $\ran \partial (g^* +\scal{x_{0}}{\cdot})=A+x_{0}=C$.
\end{proof}

\begin{theorem}\label{t:closedset}
Every nonempty closed convex set $C\subseteq \RR^n$ is a domain or
range of a subdifferential mapping of a proper lower semicontinuous
convex function.
\end{theorem}
\begin{proof}
The proper lower semicontinuous convex function
$\iota_{C}$ has $\partial \iota_{C}=N_{C}$ so that
$\dom \partial \iota_{C}=C.$
Its Fenchel conjugate  $\sigma_{C}:=\iota_{C}^*$
has $\ran \partial\sigma_{C}=\ran(\partial \iota_{C})^{-1}=C.$
\end{proof}

\begin{corollary}
Assume that $(\forall i= 1, \ldots, m)\quad O_{i}\subseteq \RR^n$ is open convex, that $
(\forall j=1,\ldots, k) \quad F_{j}\subseteq \RR^n$ is closed convex, and that
\begin{equation}\label{e:cq}
(\cap_{i=1}^{m} O_{i})\bigcap  (\cap_{j=1}^{k}\reli F_{j})\neq\varnothing.
\end{equation}
 Then
there exists a lower semicontinuous convex function $f:\RR^n\rightarrow \RX$ such that
$\dom \partial f= (\cap_{i=1}^{m} O_{i})\bigcap (\cap_{j=1}^{k}F_{j})$, equivalently,
$\ran \partial f^*=(\cap_{i=1}^{m} O_{i})\bigcap (\cap_{j=1}^{k}F_{j})$.
\end{corollary}
\begin{proof} Without loss of generality, we can assume $0\in (\cap_{i=1}^{m} O_{i})\bigcap \reli (\cap_{j=1}^{k}F_{j})$.
For each $O_{i}$, as in Theorem~\ref{t:openset}, define a lower semicontinuous convex function
$$g_{i}(x):=\begin{cases}
\frac{1}{1-\rho_{O_{i}}(x)} & \text{ if $x\in O_{i}$,}\\
+\infty & \text{ otherwise.}
\end{cases}
$$
For each $F_{j}$, as in Theorem~\ref{t:closedset}, define a lower semicontinuous convex function
$\iota_{F_{j}}$.  For the lower semicontinuous convex function
$$f:=g_{1}+\cdots+g_{m}+\iota_{F_{1}}+\cdots+\iota_{F_{k}},$$
the constraint qualification \eqref{e:cq} guarantees that the subdifferential sum rule applies, see, e.g.,
\cite[Theorem 23.8]{Rock70}
or \cite[Corollary 16.39]{BC2011}. This gives
$$\partial f=\partial g_{1}+\cdots+\partial g_{m}+\partial \iota_{F_{1}}+\cdots +\partial \iota_{F_{k}}.$$
Therefore, the subdifferential operator $\partial f$ has
$$\dom \partial f=
(\cap_{i=1}^{m}\dom \partial g_{i})\bigcap (\cap_{j=1}^{k}\dom \partial \iota_{F_{j}})
=(\cap_{i=1}^{m} O_{i})\bigcap (\cap_{j=1}^{k}F_{j}).$$
\end{proof}

\subsection{Nearly convex sets in $\RR$}
Suppose that $C$ is a nonempty nearly convex subset of $\RR$.
Then either $C$ is a singleton or $C$ is an interval,
and consequently $C$ is convex.

\begin{theorem}
Suppose that $C$ is a nonempty nearly convex
(hence convex) subset of $\RR$.
Then there exists a proper lower semicontinuous convex function
$f:\RR\rightarrow\RX$
such that $\dom \partial f =C$. Consequently,
$\ran \partial f^*=C$.
\end{theorem}
\begin{proof}
We argue by cases.\\
Case (i): $C$ is closed. This covers $C:=(-\infty,+\infty), (-\infty, b], [a, +\infty), [a,b]$
where $a, b\in\RR$.  We let $f:=\iota_{C}$. Then
$\partial f=N_{C}$ has $\dom \partial f =C$.

Case (ii): $C:=(a, +\infty)$ or $(-\infty, b)$. We only consider $C=(a,+\infty)$, since the arguments
for $C:=(-\infty,b)$ is similar. Let
$$f(x):=\begin{cases}
\frac{1}{x-a} &\text{ if $x>a$,}\\
+\infty &\text{ otherwise.}
\end{cases}
$$
Then
$$\partial f(x)=\begin{cases}
-\frac{1}{(x-a)^2} &\text{ if $x>a$,}\\
\varnothing &\text{ otherwise,}
\end{cases}$$
has
$\dom \partial f=(a,+\infty)$.

In the remaining cases, we can and do assume $a,b\in\RR$.

Case (iii): If $C:=]a,b[$ with $a,b\in \RR$ and $a<b$, we put
$$f(x):=\begin{cases}
-\frac{b-a}{\pi}\ln\cos\bigg(\frac{\pi(x-a)}{b-a}-\frac{\pi}{2}\bigg) & \text{ if $a<x<b$,}
\\
+\infty &\text{ otherwise.}
\end{cases}
$$
Then
$$\partial f(x)=
\begin{cases}
\tan\bigg(\frac{\pi}{b-a}(x-a)-\frac{\pi}{2}\bigg) & \text{ if $a<x<b$,}\\
\varnothing &\text{ otherwise.}
\end{cases}
$$
has $\dom \partial f=(a,b)$.

Case (iv): $C$ is half-open interval. Suppose without loss of generality that
$C:=]a,b]$ with $a,b\in \RR$ and $a<b$.
We put
$$f(x):=\begin{cases}
-\big(\ln(x-a)-\frac{x}{(b-a)}\big) & \text{ if $a<x\leq b$},\\
+\infty & \text{ otherwise.}
\end{cases}
$$
Then
\begin{equation}
\partial f(x)=\begin{cases}
\frac{1}{a-x}-\frac{1}{a-b}& \text{if } a<x<b,\\
[0,+\infty[ & \text{if } x=b,\\
\varnothing & \text{ otherwise.}\\
\end{cases}
\end{equation}
has $\dom \partial f=(a,b]$.
\end{proof}

\subsection{Nearly convex sets in $\RR^2$}
We start with a complete analysis of the classical example by Rockafellar \cite[page 218]{Rock70}.
Often in literature, it only gives that $\dom \partial f$ is not convex without details. His function
is modified for the
convenience
of our later constructions. The set of nonpositive real numbers is $\RR_{-}:=\menge{x\in\RR}{x\leq 0}$.
\begin{example}\label{Rock:ex:gen}
Let $\alpha >0$ and define $f(\xi_1,\xi_2):\RR^2\to \RX$ by
\begin{equation}
f(\xi_1,\xi_2):=\begin{cases}
\max\stb{\alpha-{\xi_1}^{\tfrac{1}{2}},\abs{\xi_2}} & \text{if }\xi_1\ge 0,\\
+\infty  & \text{otherwise.}
\end{cases}
\end{equation}
Then $\pt f(\xi_1,\xi_2)=$
\begin{equation}\label{eq:sub:cases}
\begin{cases}
\fady &\text{ if $\xi_{1}<0$,}\\
\fady   & \text{if } \xi_1 =0, \text { and }\abs{\xi_2}<\alpha,\\
\RR_{-}\times\stb{1}    &  \text{if }  \xi_1=0, \text { and } \xi_2\ge \alpha, \\
\RR_{-}\times\stb{-1}   &  \text{if }  \xi_1=0,  \text{ and }\xi_2\le -\alpha, \\
\conv\stb{(-\tfrac{1}{2}{\xi_1}^{-{1}/{2}},0),(0,1)}
&  \text{if } {\xi_2}= \alpha-\sqrt{\xi_1}, \text { and } 0<\xi_1<\alpha^2,\\
\conv\stb{(-\tfrac{1}{2}{\xi_1}^{-{1}/{2}},0),(0,-1)}
&  \text{if } {-\xi_2}= \alpha-\sqrt{\xi_1}, \text { and } 0<\xi_1<\alpha^2,\\
(-\tfrac{1}{2}{\xi_1}^{-\tfrac{1}{2}},0) &  \text{if } 0<\xi_1<\alpha^2,
 \text { and }\alpha-\sqrt{\xi_1}>\abs{\xi_2},\\
(0,1) &  \text{if $0<\xi_{1}<\alpha^2$, } \text{ and }  \xi_2> \alpha-\sqrt{\xi_1},\\
(0,-1) &  \text{if $0<\xi_{1}<\alpha^2$, } \text{ and } -\xi_2 > \alpha-\sqrt{\xi_1},\\
\conv\stb{(-\tfrac{1}{2\alpha},0),(0,1),(0,-1)} &  \text{if }  \xi_1=\alpha^2, \text { and } \xi_2=0, \\
\conv\stb{(0,1),(0,-1)}
&   \text { if }\xi_1>\alpha^2, \text{ and } \xi_2=0,\\
(0,1) &  \text{if $\xi_{1}>\alpha^2$, and }   \xi_2>0,\\
(0,-1) &  \text{if $\xi_{1}>\alpha^2$, and }  -\xi_2 > 0.
\end{cases}
\end{equation}
Consequently,
\begin{equation}\label{randf}
\ran \pt f=\stb{(\xi_1,\xi_2)~|~\xi_1\le 0, \abs{ \xi_2}\le 1},
\end{equation}
and the domain of $\partial f$
\begin{align}
\dom \pt f&=\stb{(\xi_1,\xi_2)~|~\xi_1> 0,  \xi_2\in \RR}\cup
\stb{(0,\xi_2)~|~\abs{\xi_2}\ge \alpha}\label{domdf}
\end{align}
is almost convex but not convex.
Moreover, the Fenchel conjugate of $f$ is $f^*(x_1^*,x_2^*) = $
\begin{equation}
\begin{cases}
0 & \text{if }\xso\le 0 \text{ and }\abs{\xst}=1, \text{~or~}\xso= 0 \text{ and }\abs{\xst}<1,\\
\alpha^2 \xso & \text{ if $\tfrac{-1}{2\alpha}\leq \xso\leq 0$, and $|\xst|\leq 1+2\alpha\xso$,}\\
-\tfrac{(1-\abs{\xst})^2}{4{\xso}}
-\alpha(1-\abs{\xst}) & \text{if }\xso<0, \text{~and~} \max\{0,1+2\alpha \xso\}\leq \xst \leq 1,\\
-\tfrac{(1-\abs{\xst})^2}{4{\xso}}
-\alpha(1-\abs{\xst}) & \text{if }\xso<0, \text{~and~} -1\leq \xst \leq \min\{0,-(1+2\alpha \xso)\},\\
+\infty & \text{ otherwise.}
\end{cases}
\end{equation}
\end{example}
See Appendix A for its proof.  A different example is given in \cite{bbmw15}.
\begin{figure}[H]
\begin{center}
\begin{tabular}{c c c}
%\hline
\includegraphics[scale=0.5]{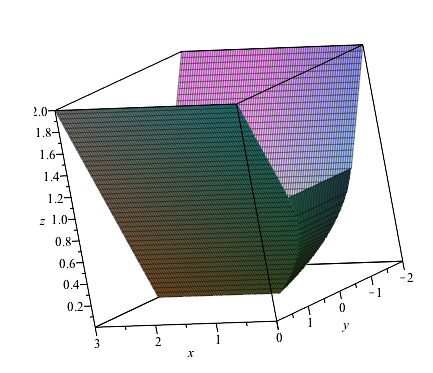}    & &
\includegraphics[scale=0.5]{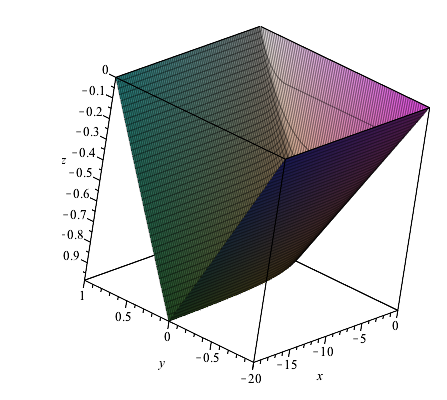}   \\
\end{tabular}
\end{center}
 \caption{A \texttt{Maple} \cite{maple} snapshot.
Left: Plot of $f$ with $\alpha=1$. Right: Plot of $f^*$ with $\alpha=1$.}
\label{fig:f:f*}
\end{figure}

\begin{example}\label{e:notclosed} Let $\alpha\geq 0$.
Define $f:\RR^2\rightarrow\RX$ by
$$f(x_{1},x_{2}):=\begin{cases}
\max\{\alpha-\sqrt{x_{1}},x_{2}\} & \text{ if $x_{1}\geq 0$},\\
+\infty & \text{ otherwise.}
\end{cases}
$$
Then
$$\partial f(x_{1},x_{2})=
\begin{cases}
(-1/2 x_{1}^{-1/2},0) & \text{ if $x_{1}> 0$, and $x_{2}< \alpha-\sqrt{x_{1}}$,}\\
\conv\{(0,1), (-1/2 x_{1}^{-1/2},0)\} &\text{ if $x_{1}> 0$, and $x_{2}=\alpha-\sqrt{x_{1}}$,}\\
\{(0,1)\} &\text{ if $x_{1}> 0$, and $x_{2}>\alpha-\sqrt{x_{1}}$,}\\
\varnothing &\text{ if $x_{1}=0$, and $x_{2}<\alpha$,}\\
\RR_{-}\times\{1\} &\text{ if $x_{1}=0$, and $x_{2}\geq \alpha$,}\\
%\RR_{-}\times \{1\} &\text{ if $x_{2}>\alpha$ and $x_{1}=0$;}\\
\varnothing & \text{ if $x_{1}<0$.}
\end{cases}
$$
In particular,
$\dom \partial f=\menge{(x_{1},x_{2})}{x_{1}\geq 0}\setminus\menge{(0,x_{2})}{x_{2}<\alpha}$
is neither open nor closed, but it is convex. The same holds for the range
$$\ran\partial f=\menge{(x_{1},x_{2})}{x_{1}\leq 0, 0\leq x_{2}\leq 1}\setminus
\menge{(0,x_{2})}{0\leq x_{2}<1}.$$
\end{example}

See Appendix A for its proof.

We are finally positioned to construct proper lower semicontinuous convex functions on $\RR^2$
whose subdifferential
domains are nearly polyhedral sets but nonconvex.
Recall that $C\subseteq\RR^n$ is said to be polyhedral if it can be expressed as the intersection of a
finite family closed half spaces or hyperplanes, i.e.,
$$C=\big(\cap_{i=1}^{m}\menge{x\in\RR^n}{f_{i}(x)\leq 0}\big)\bigcap
\big(\cap_{j=1}^{k}\menge{x\in\RR^n}{f_{j}(x)=0}\big)$$
where each $f_{i}$ is affine.

\begin{theorem}\label{thm:polygon:noedge} In $\RR^2$, every polyhedral set $C$ having a nonempty interior and
having edges removed but keeping all its vertices is a domain of a subdifferential mapping
 of a proper
lower
semicontinuous convex function on $\RR^2$.
\end{theorem}
\begin{proof}
Each polyhedral set is closed and convex \cite[Example 2.10]{Rock98}. Associated each edge $[x_{i},x_{i+1}]$ of
$C$, one can find
a closed half space $H_{i}\subseteq \RR^2$ contains $C$ and has $[x_{i},x_{i+1}]$ in its boundary.
By using translation, rotation and dilation for the convex function in Example~\ref{Rock:ex:gen},
we can get a lower semicontinuous convex function
$f_{i}:\RR^2\rightarrow\RX$ such that $\dom \partial f_{i}=H_{i}\setminus ]x_{i},x_{i+1}[$. For each edge of the form
emanating from $x_{j}$ in the direction $v_{j}$: $R_{j}=\menge{x_{j}+\tau v_{j}}{\tau\geq 0}$,  one can find
a closed half space $H_{j}\subseteq\RR^2
$ containing $C$ and has $R_{j}$ in its boundary. Again, by using
translation, rotation, and dilation for the convex function in Example~\ref{e:notclosed},
we get a lower semicontinuous convex function
$f_{j}:\RR^2\rightarrow\RX$ such that
$\dom \partial f_{j}=H_{j}\setminus \menge{x_{j}+\tau v_{j}}{\tau > 0}$.
Define the lower semcontinuous convex function $f:\RR^2\rightarrow\RX$ by
$$f:=\iota_{\closu{C}}+\sum_{i\in I}f_{i}+\sum_{i\in J}f_{j}.$$
As $\inte C\neq\varnothing$, by the subdifferential sum rule we have
$\partial f=\partial\iota_{\closu{C}}+\sum_{i\in I}\partial f_{i}+\sum_{i\in J}\partial f_{j}.$
The maximally monotone operator $\partial f$ has
$\dom \partial f =C$.
\end{proof}

Immediately from Theorem~\ref{thm:polygon:noedge}, we see that
each set in Figure~\ref{fig:nc:set} is the subdifferential domain
of a proper lower semicontinuous convex function on $\RR^2$.
\begin{figure}[H]
\begin{center}
\begin{tabular}{c c c}
%\hline
\includegraphics[scale=0.5]{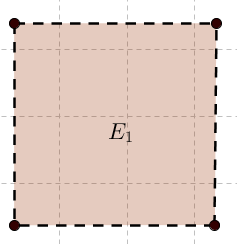}    & &
\includegraphics[scale=0.5]{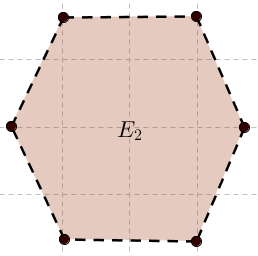}   \\
\includegraphics[scale=0.5]{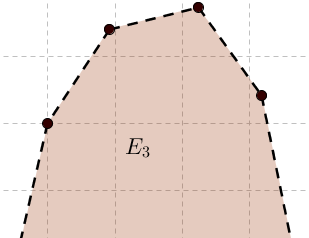}    & &
\includegraphics[scale=0.5]{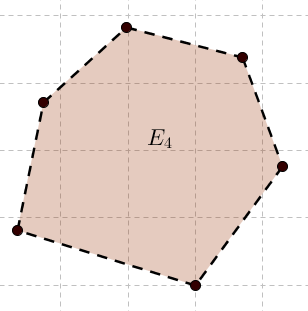}   \\
\end{tabular}
\end{center}
 \caption{A \texttt{GeoGebra} \cite{geogebra} snapshot.
Nearly convex but not convex sets}
\label{fig:nc:set}
\end{figure}

As a concrete example, we have

\begin{example}\label{ncpolygon}
Suppose that $f:\RR^2\to\RX$ is defined as
$f:=\sum_{i \in I}f_i$,
where $I:=\{1, 2, 3, 4\}$ and $(\forall i\in I)$ $f_i$ are defined as
\begin{align}
f_{1}&:\RR^2\to\RR:(x,y)\mapsto \max\stb{\sqrt{2}-\sqrt{\tfrac{1}{\sqrt{2}}(2+x-y)},\abs{\tfrac{1}{\sqrt{2}}(2+x+y)}},\\
f_{2}&:\RR^2\to\RR:(x,y)\mapsto  \max\stb{3-\sqrt{1+y},\abs{x}},\\
f_{3}&:\RR^2\to\RR:(x,y)\mapsto  \max\stb{1-\sqrt{1-y},\abs{x}} ,\\
f_{4}&:\RR^2\to\RR:(x,y)\mapsto  \max\stb{\sqrt{2}-\sqrt{\tfrac{1}{\sqrt{2}}(2-x-y)},\abs{\tfrac{1}{\sqrt{2}}(-2+x-y)}}.
\end{align}
Then $f$ is a proper, convex, lower semicontinuous function and
\begin{gather*}
\dom \pt f=\ran (\pt f)^{-1}=\ran \pt f^*\\
=\stb{(x,y)~|~\abs{x}<3,\abs{y}<1,2+x-y> 0,2-x-y>0}\cup\stb{(1,1),(-1,1),(3,-1),(-3,-1)},
\end{gather*}
as shown in Figure~\ref{fig:ncpolygon}.
\end{example}
\begin{figure}[h]
  \centering
  \begin{subfigure}[t]{0.49\textwidth}
    \centering
    \includegraphics[scale=0.5]{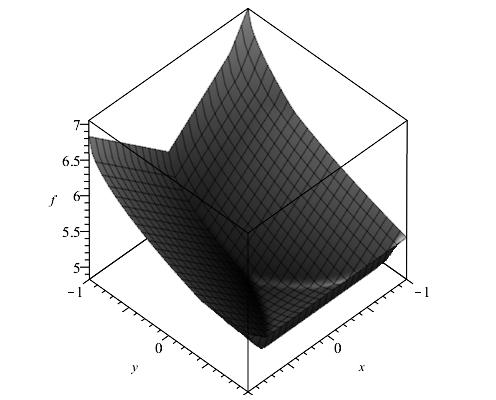}
    \caption{$f(\xi_1,\xi_2)$}
  \end{subfigure}
  \begin{subfigure}[t]{0.49\textwidth}
    \centering
    \includegraphics[scale=0.5]{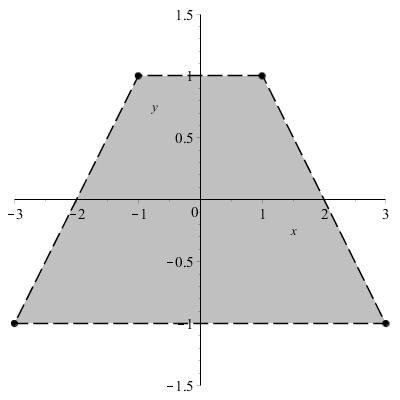}
    \caption{$\dom \pt f$}
  \end{subfigure}
  \caption{The function $f$ and $\dom \pt f$ of Example~\ref{ncpolygon}.}
  \label{fig:ncpolygon}
\end{figure}

\begin{proof}
Let
$$g_\alpha:\RR^2\to \RR:(x,y)\mapsto \begin{cases} \max\stb{\alpha-\sqrt{x},\abs{y}}, &\text{ if }x\ge 0;\\ +\infty, & \text{otherwise,}\end{cases}$$
and
$$R_{\theta} = \begin{pmatrix} \cos \theta & -\sin \theta\\ \sin \theta & \cos \theta\end{pmatrix}.$$
 Then $(\forall (x,y)\in \RR^2)$ we have,
\begin{align*}
f_1(x,y)&=g_{\sqrt{2}}(R_{\pi/4}((x,y)-(-2,0))),\\
f_2(x,y)&=g_3(R_{-\pi/2}((x,y)-(0,-1))),\\
f_3(x,y)&=g_1(R_{\pi/2}((x,y)-(0,1))),\\
f_4(x,y)&=g_{\sqrt{2}}(R_{5\pi/4}((x,y)-(2,0))).
\end{align*}
As $\pt f=\sum_{i \in I} \pt f_i$,
 using Example~\ref{Rock:ex:gen}, and particularly \eqref{domdf},
we see that
 \begin{gather*}
\dom \pt f=\cap_{i \in I} \dom \pt f_i\\
=\stb{(x,y)~|~\abs{x}<3,\abs{y}<1,2+x-y> 0,2-x-y>0}\cup\stb{(1,1),(-1,1),(3,-1),(-3,-1)}.
\end{gather*}
Using \cite[Proposition~16.24]{BC2011} we have $(\pt f)^{-1}=\pt f^*$,
which completes the proof.
\end{proof}

We finish this section by remarking that each set in Figure~\ref{fig:c:set} is a subdifferential domain
of a proper lower semicontinuous convex function. For the first set, use
$f:=\iota_{B(0,1)}+g_{1}+g_{2}$  where each $g_{i}$ has $\dom \partial g_{i}$ being an open convex set whose boundary consists
of one dotted line part of the unit circle and two dotted rays tangent to the circle.
For the second set, use $f:=\iota_{B(0,1)}+g_{1}+g_{2}$ where each $g_{i}$ is obtained by rotation and translation
of Rockafellar's function, and $\dom \partial g_{i}$ is a closed half space with an open line segment on its boundary removed.

\section{Discussions and open problems}\label{s:openproblems}
In this paper we systematically study nearly convex sets: criteria for near convexity; topological properties such as
relative interior, interior, recession cone
of nearly convex sets; formulas for the relative interiors and closures of nearly convex sets, which are linear
image or inverse image of other nearly convex sets. Rockafellar provided the first convex function whose subdifferential
domain is not convex. To build more examples, we compute the subdifferential and the Fenchel
conjugate of an modified Rockafellar's function.
It turns out every polyhedral set in $\RR^2$ with edges removed but keeping its vertices
is a domain of the subdifferential mapping of a proper lower semicontinuous convex function.

Although we have constructed some proper lower semicontinuous convex functions whose subdifferential
mappings have prescribed domain or ranges, the general problem is still unsolved.
Let us note that
\begin{theorem} Let $C\subseteq \RR^n$ (not necessarily convex).
Then there exists a monotone operator (not necessarily maximal monotone)
 $A:\RR^n\To \RR^n$ such that
$\ran A=C$.
\end{theorem}
\begin{proof} Consider the projection operator $P_{C}:\RR^n\To\RR^n$. Then $P_{C}$ is
monotone because $\gra P_{C}\subseteq \gra P_{\closu{C}}$, and $P_{\closu{C}}$ is monotone by \cite[Proposition 12.19]{Rock98}.
Clearly $\ran P_{C}=C$.
\end{proof}

According to \cite[Theorem 12.20]{Rock98}, a closed set $C\subseteq\RR^n$ is convex if and only if
$P_{C}$ is maximally monotone. Therefore, it is the maximality to force $C$ having more structural
properties.

We finish the paper with three open questions:
\begin{enumerate}
\item Is every convex set a domain or range of a subdifferential mapping of a proper lower semicontinuous convex function (or
a maximally monotone operator) in $\RR^n$ with $n\geq 2$?
\item Is every nearly convex set a domain or range of a subdifferential mapping of a proper lower semicontinuous convex function (or a maximally monotone operator) in $\RR^n$ with $n\geq 2$?
\item What is the intrinsic difference between the ranges of subdifferentials of proper lower semicontinuous convex functions
and the ranges of maximally monotone operators?
\end{enumerate}

\section*{Acknowledgments}
Sarah Moffat was partially supported by the Natural Sciences and Engineering Research
Council of Canada. Walaa Moursi was supported by NSERC grants of Drs. Heinz Bauschke and Warren Hare.
Xianfu Wang was partially supported by the Natural Sciences and
Engineering Research Council of Canada.

\section*{Appendix A}\label{a:missingproofs}
% \newpage
We shall need two facts. The first one is a structural
characterization of $\partial f$ when a convex function $f$ is lower
semicontinuous, and $\nabla f$ is not empty. (See also \cite[Theorem 12.67]{Rock98}
for the structure of maximally monotone operators.) The second one concerns the domain of the Fenchel
conjugate of convex functions.

\begin{fact}\emph{(\cite[Theorem 25.6]{Rock70})}\label{f:subcharacter}
Let $f$ be a closed proper convex function such that $\dom f$ has a non-empty interior.
Then
$$\partial f(x)=\closu(\conv S(x))+K(x), \quad \forall x,$$
where
$K(x)$ is the normal cone to $\dom f$ at $x$ (empty if $x\not\in\dom f$)
and $S(x)$ is the set of all limits of sequence of the form
$\nabla f(x_{1}), \nabla f(x_{2}), \ldots,$ such that
$f$ is differentiable at $x_{i}$, and $x_{i}$ tends to $x$.
\end{fact}
\begin{lemma}\label{l:conj:dom}
Assume that $f:\RR^n\rightarrow\RX$ is proper lower semicontinuous convex function.
If $\ran\partial f$ is closed, then
$\dom f^*=\ran\partial f$.
\end{lemma}
\begin{proof} As $\ran \partial f=\dom \partial f^*$, by the Bronsted-Rockafellar's theorem
\cite[Theorem 3.17]{phelps} or \cite[Proposition 16.28]{BC2011}, we obtain
$\ran \partial f \subseteq \dom f^*\subseteq \closu(\ran\partial f).$
Therefore, the result holds.
\end{proof}

\noindent{\bf I. Proof of Example~\ref{Rock:ex:gen}}

\begin{proof}
First, we calculate $\partial f$.
We argue by cases:
\begin{enumerate}
\item\label{i:differentiable}
$\alpha-\sqrt{\xi_1}>\abs{\xi_2}$, $0<\xi_1<\alpha^2$: $\nabla f(\xi_1,\xi_2)=\stb{(-\tfrac{1}{2}{\xi_1}^{-1/2},0)}$
because $f(\xi_1,\xi_2)=\alpha -{\xi_1}^{1/2}$.

\item\label{i:insideregion} $\xi_1 =0$, $\abs{\xi_2}<\alpha$: $\pt f(0,\xi_2)=\fady$.
Indeed, by \ref{i:differentiable},
$\lim_{(x_{1},x_{2})\rightarrow (0,\xi_{2})}\nabla f(x_{1},x_{2})=\lim_{(x_{1},x_{2})\rightarrow (0,\xi_{2})}(-\tfrac{1}{2}{x_1}^{-1/2},0)$
does not exists. Apply Fact~\ref{f:subcharacter}.

\item\label{i:nonempty:above} $\xi_1=0$, $\xi_2\ge \alpha$: $\pt f(0,\xi_2) = \RR_{-}\times\stb{1}$.
Note that $N_{\dom f}(0,\xi_{2})=\RR_{-}\times\{0\}$. When $x_2> \alpha-\sqrt{x_1}$, $\alpha^2>x_1>0$,
we have $f(x_{1},x_{2})=x_{2}$, so  $$\lim_{(x_{1},x_{2})\rightarrow (0,\xi_{2})}
\nabla f(x_{1},x_{2})=\{(0,1)\};$$
When $\alpha-\sqrt{x_{1}}>x_2$, $0<x_1<\alpha^2$, $x_{2}>0$,
we have $f(x_1,x_2)=\alpha -{x_1}^{1/2}$, so
$$\lim_{(x_{1},x_{2})\rightarrow (0,\xi_{2})}\nabla f(x_1,x_2)
=\lim_{(x_{1},x_{2})\rightarrow (0,\xi_{2})}(-\tfrac{1}{2}{x_1}^{-1/2},0)$$
does not exist.
Apply Fact~\ref{f:subcharacter} to obtain $\partial f(0,\xi_{2})$.

\item $\xi_1=0$, $\xi_2\le -\alpha$: $\pt f(0,\xi_2) = \RR_{-}\times\stb{-1}$. The arguments
are similar to \ref{i:nonempty:above}.

\item  $\xi_2> \alpha-\sqrt{\xi_1}$, $\alpha^2>\xi_1>0$: $\partial f(\xi_{1},\xi_{2})=\{(0,1)\}$
because $f(\xi_{1},\xi_{2})=\xi_{2}$.

\item $\xi_2= \alpha-\sqrt{\xi_1}$, $\alpha^2>\xi_1>0$:
Notice that $f(\xi_1,\xi_2)=\max\stb{f_1(\xi_1,\xi_2),f_2(\xi_1,\xi_2)}$, where
$f_1(\xi_1,\xi_2)=\alpha-\sqrt{\xi_1}$ and $f_2(\xi_1,\xi_2)=\xi_2$.
When $\abs{\xi_2}= \alpha-\sqrt{\xi_1}$, $\xi_1>0$, and $\xi_2>0$, we have
 $f_1(\xi_1,\xi_2)=f_2(\xi_1,\xi_2)$, hence
it follows from \cite[Theorem~10.31]{Rock98} that
$$\pt{f(\xi_1,\xi_2)}=\conv\stb{\grad f_1(\xi_1,\xi_2),\grad f_1(\xi_1,\xi_2)} =\conv\stb{(-\tfrac{1}{2}{\xi_1}^{-\tfrac{1}{2}},0),(0,1)}.$$

\item $-\xi_2= \alpha-\sqrt{\xi_1}$, $\alpha^2>\xi_1>0$:
The proof is similar to the previous case with
$f(\xi_1,\xi_2)=\max\stb{f_1(\xi_1,\xi_2),f_2(\xi_1,\xi_2)}$, where
$f_1(\xi_1,\xi_2)=\alpha-\sqrt{\xi_1},$ and $f_2(\xi_1,\xi_2)=-\xi_2$, thus
$$\pt{f(\xi_1,\xi_2)}=\conv\stb{(-\tfrac{1}{2}{\xi_1}^{-{1}/{2}},0),(0,-1)}.$$

\item $\xi_1=\alpha^2$, $\xi_2=0$:
We have
$f(\xi_1,\xi_2)=\max\stb{f_1(\xi_1,\xi_2),f_2(\xi_1,\xi_2),f_3(\xi_1,\xi_2)}$, where
$f_1(\xi_1,\xi_2)=\alpha-\sqrt{\xi_1},$ $f_2(\xi_1,\xi_2)=\xi_2$ and $f_3(\xi_1,\xi_2)=-\xi_2$, thus
$$\pt{f(\xi_1,\xi_2)}=\conv\stb{(-\tfrac{1}{2\alpha},0),(0,1),(0,-1)}.$$

\item
When $\xi_{1}>\alpha^2$, $f(\xi_1,\xi_2)=\abs{\xi_2}=\max\stb{\xi_2,-\xi_2}$.
If $\xi_{2}>0$, then
$f(\xi_1,\xi_2)=\xi_2$, so
$\pt f(\xi_1,\xi_2)=\stb{(0,1)}$.
If $\xi_2<0$, then
$f(\xi_1,\xi_2)=-\xi_2$, so
$\pt f(\xi_1,\xi_2)=\stb{(0,-1)}$. If $\xi_{2}=0$,  then $\partial f(\xi_{1},0)=\conv\{(0,1),(0,-1)\}
=\{0\} \times [-1,1]$.
\end{enumerate}

Next, we calculate  the  Fenchel conjugate of $f$.

In view of \eqref{randf}, $\ran \partial f$ is closed,
so  $\dom f^*=\ran \partial f$ by Lemma~\ref{l:conj:dom}.
Recall that
\begin{equation}\label{e:FY:dual}
f^*(\xso,\xst)=x_1\xso+x_2\xst-f(x_1,x_2)
\end{equation}
We proceed by cases using \eqref{eq:sub:cases}.
\begin{enumerate}[(i)]
\item \label{i:conj:one}
$\xso\le0$ and $\xst=1$:
In this case $x_1=0$ and $\abs{x_2}=x_2\ge \alpha$,
hence $f(x_1,x_2)=x_2$.
Therefore \eqref{e:FY:dual} implies that
$f^*(\xso,\xst)=x_2-x_2=0$.

\item
$\xso\le0$ and $\xst=-1$:
In this case $x_1=0$ and $\abs{x_2}=-x_2\ge \alpha$,
hence $f(x_1,x_2)=-x_2$.
Therefore \eqref{e:FY:dual} implies that
$f^*(\xso,\xst)=-x_2+x_2=0$.

\item
$-1/(2\alpha)\leq \xso\leq 0$ and $|\xst|\leq 1+2\alpha \xso$:
 In this case, this is exactly the region given by the set
$\conv\stb{(-\tfrac{1}{2\alpha},0),(0,1),(0,-1)}=\partial f(\alpha^2,0)$.
Hence, by \eqref{eq:sub:cases} we have
$(\xso,\xst)\in \partial f(\alpha^2,0)$ so that
$x_1=\alpha^2$, $x_2=0$ and
$f(x_1,x_2)=0$. Therefore,
\begin{align}
f^*(\xso,\xst)&=x_1\xso+x_2\xst-f(x_1,x_2)\nonumber\\
&=\alpha^2\cdot \xso+0\cdot\xst-0=\alpha^2\xso.
\end{align}

\item $\xso< 0$ and $\max\{0, 1+2\alpha \xso\}\le\xst< 1$:
In this case, this is the region given by
$$\bigcup\left\{\conv\{(-1/2 x_{1}^{-1/2},0), (0,1)\}: 0<x_{1}<\alpha^2, x_{2}=\alpha-x_{1}^{1/2}\right\}\setminus
\{(0,1)\}.$$
 Then  each $(\xso,\xst)\in
\conv\stb{(-\tfrac{1}{2\sqrt{x_1}},0),(0,1)}$ for some $(x_{1},x_{2})$ satisfying
$0<x_{1}<\alpha^2, x_{2}=\alpha-x_{1}^{1/2}$.
Thus, there exists $\lam\in \left]0,1\right]$ such that
$\xso=-\tfrac{\lam}{2\sqrt{x_1}}$
 and $\xst=1-\lam$. Therefore we have
 $\tfrac{1}{\sqrt{x_1}}=-\tfrac{2}{\lam}\xso
 =-\tfrac{2}{1-\xst}\xso$,
 $\sqrt{x_1}=-\tfrac{1-\xst}{2\xso}$,
  $x_2=\alpha-\sqrt{x_1}=\alpha+\tfrac{1-\xst}{2\xso}$,
and $f(x_1,x_2)=\alpha-\sqrt{x_1}=x_2$.
Now  \eqref{e:FY:dual} implies that
\begin{align}
f^*(\xso,\xst)&=x_1\xso+x_2\xst-f(x_1,x_2)\nonumber\\
&=\tfrac{(1-\xst)^2}{4{\xso}^2}\xso
+(\alpha+\tfrac{1-\xst}{2\xso})\xst
-(\alpha+\tfrac{1-\xst}{2\xso})\nonumber\\
&=\tfrac{(1-\xst)^2}{4{\xso}}
+(\alpha+\tfrac{1-\xst}{2\xso})(\xst-1)
=\tfrac{(1-\xst)^2}{4{\xso}}-\tfrac{(1-\xst)^2}{2{\xso}}
-\alpha(1-\xst)
\nonumber\\
&=-\tfrac{(1-\xst)^2}{4{\xso}}
-\alpha(1-\xst).
\nonumber
\end{align}

\item\label{i:conj:five} $\xso< 0$ and $-1<\xst \leq \min\{0, -(1+2\alpha \xso)\}$:
In this case, this is the region given by
$$\bigcup\left\{\conv\{(-1/2 x_{1}^{-1/2},0), (0,-1)\}: 0<x_{1}<\alpha^2, -x_{2}=\alpha-x_{1}^{1/2}\right\}
\setminus\{(0,-1)\}.$$
Then each
$(\xso,\xst)\in
\conv\stb{(-\tfrac{1}{2\sqrt{x_1}},0),(0,-1)}$ for some $(x_{1},x_{2})$ satisfying
$0<x_{1}<\alpha^2, -x_{2}=\alpha-x_{1}^{1/2}$.
As before,
let $\lam\in \left]0,1\right]$.
Then $\xso=-\tfrac{\lam}{2\sqrt{x_1}}$
 and $\xst=-(1-\lam)=\lam-1$. Therefore we have
 $\tfrac{1}{\sqrt{x_1}}=-\tfrac{2}{\lam}\xso
 =-\tfrac{2}{1+\xst}\xso$,
 $\sqrt{x_1}=-\tfrac{1+\xst}{2\xso}$,
  $x_2=-(\alpha-\sqrt{x_1})=-\alpha-\tfrac{1+\xst}{2\xso}$,
and $f(x_1,x_2)=\alpha-\sqrt{x_1}=-x_2$.
Now  \eqref{e:FY:dual} implies that
\begin{align}
f^*(\xso,\xst)&=x_1\xso+x_2\xst-f(x_1,x_2)\nonumber\\
&=\tfrac{(1+\xst)^2}{4{\xso}^2}\xso
-(\alpha+\tfrac{1+\xst}{2\xso})\xst
-(\alpha+\tfrac{1+\xst}{2\xso})\nonumber\\
&=\tfrac{(1+\xst)^2}{4{\xso}}
-(\alpha+\tfrac{1+\xst}{2\xso})(\xst+1)
=\tfrac{(1+\xst)^2}{4{\xso}}-\tfrac{(1+\xst)^2}{2{\xso}}
-\alpha(1+\xst)
\nonumber\\
&=-\tfrac{(1+\xst)^2}{4{\xso}}
-\alpha(1+\xst)=-\tfrac{(1-\abs{\xst})^2}{4{\xso}}
-\alpha(1-\abs{\xst}).
\end{align}
\eqref{i:conj:one}-\eqref{i:conj:five} together finish the computation of $f^*$.
\end{enumerate}
Altogether,  the proof is complete.
\end{proof}

\noindent{\bf II. Proof of Example~\ref{e:notclosed}}

\begin{proof} We argue by cases.

Case 1: $x_{2}< \alpha-\sqrt{x_{1}}$ and $x_{1}\geq 0$.
We have $f(x_{1},x_{2}):=\alpha-\sqrt{x_{1}}$.
When $x_{1}>0$,
$f(x_{1},x_{2})=\alpha-\sqrt{x_{1}}$
so $\partial f(x_{1},x_{2})=(-1/2 x_{1}^{-1/2},0)$;
When $x_{1}=0$ and $x_{2}<\alpha$, $f(0,x_{2})=\alpha$, $f(x_{1},x_2)=\alpha-\sqrt{x_{1}}$ when $x_{1}>0$, so
$\partial f(0,x_{2})=\varnothing$.

Case 2: $x_{2}>\alpha-\sqrt{x_{1}}$ and $x_{1}\geq 0$. When
$x_{1}>0$, $f(x_{1},x_{2})=x_{2}$,
$\partial f(x_{1},x_{2})=\{(0,1)\}$; When $x_1=0$,
$\partial f(0,x_{2})= (0,1)+\RR_{-}\times\{0\}$.

Case 3: $x_{2}=\alpha-\sqrt{x_{1}}$. When $x_{1}>0$,
$\partial f(x_{1},x_{2})=\conv\{(0,1), (-1/2 x_{1}^{-1/2},0)\}$;
When $x_{1}=0, x_{2}=\alpha$, $\partial f(0,\alpha)=(0,1)+\RR_{-}\times\{0\}$.
\end{proof}


\begin{thebibliography}{99}
\bibitem{bartz07}
S. Bartz, H.H. Bauschke, J.M. Borwein, S. Reich, and X. Wang, ``
Fitzpatrick functions, cyclic monotonicity and Rockafellar's antiderivative,"
\emph{Nonlinear Anal.} 66, pp.~1198--1223, 2007.

\bibitem{bbw2007}
H.H. Bauschke, J.M. Borwein and X. Wang,
``Fitzpatrick functions and continuous linear monotone operators," \emph{SIAM J. Optim.}~18,
pp.~789--809, 2007.

\bibitem{bwy2012} H.H.\ Bauschke, X.\ Wang and L.\ Yao,
``On paramonotone monotone operators and
rectangular monotone operators," \emph{Optimization}, pp.~1--18, 2012.

\bibitem{BC2011}
H.H.\ Bauschke and P.L.\ Combettes,
\emph{Convex Analysis and Monotone Operator Theory in Hilbert Spaces},
Springer, 2011.

\bibitem{bmw12} H.H. Bauschke, S. Moffat and X. Wang,
``Firmly nonexpansive mappings and maximally monotone operators:
correspondence and duality," \emph{Set-Valued Var. Anal.}~20,  pp.~131--153, 2012.

\bibitem{bmw-nenc}
H.H.\ Bauschke, S.M.\ Moffat, and X.\ Wang,
``Near equality, near convexity, sums of maximally monotone operators,
and averages of firmly nonexpansive mappings,"
\emph{Mathematical Programming}~139, pp.~55--70, 2013.

\bibitem{bbmw15} Sedi Bartz, H.H.\ Bauschke, S.M.\ Moffat, and X.\ Wang,
``The resolvent average of monotone operators:
dominant and recessive properties,"
http://arxiv.org/pdf/1505.02718.pdf

\bibitem{Bonnans}
J.F.\ Bonnans and A.\ Shapiro, \emph{Perturbation Analysis of Optimization Problems},
Springer Series in Operations Research, Springer-Verlag, 2000.

\bibitem{lewis} J.M.\ Borwein and A.S.\ Lewis,
\emph{Convex Analysis and Nonlinear Optimization},
second edition, Springer, New York, 2006.

\bibitem{bzhu}
J.M.\ Borwein and Q.\ Zhu, \emph{Techniques of Variational Analysis},
 Springer-Verlag, New York, 2005.

\bibitem{Bot2006}
R.I.\ Bo\c{t}, S.M.\ Grad, and G. Wanka, ``Almost convex functions: conjugacy and duality," \emph{Lecture Notes in Economics and Mathematical Systems}~583, pp.~101--114, 2006.

\bibitem{Bot2007}
R.I.\ Bo\c{t}, S.M.\ Grad, and G. Wanka, ``Fenchel's duality theorem for nearly convex functions," \emph{J. Optim. Theory Appl.}~132, pp.~509--515, 2007.

\bibitem{Bot2008}
R.I.\ Bo\c{t}, G.\ Kassay, G.\ Wanka, ``Duality for almost convex optimization problema via the perturbation approach," \emph{J. Glob. Optim.}~42, pp.~385--399, 2008.


\bibitem{Br-H} H.\ Brezis, A.\ Haraux,
Image d'une Somme d'op\'{e}rateurs Monotones et Applications,
\emph{Israel J. Math.}~23, pp.~165--186, 1976.

\bibitem{Fabian} M. Fabian, P. Habala, P. H\'{a}jek, V.  Montesinos Santaluc\'{i}a, J. Pelant, and V.  Zizler,
\emph{Functional Analysis and Infinite-dimensional Geometry},
CMS Books in Mathematics/Ouvrages de Math\'{e}matiques de la SMC, 8. Springer-Verlag,
New York, 2001.

\bibitem{Frenk}
J.B.G.\ Frenk, G.\ Kassay, ``Lagrangian duality and cone convexlike functions," \emph{J. Optim. Theor. Appl.}~132(3), pp.~207--222, 2007.

\bibitem{geogebra} \texttt{http://www.geogebra.org/}.

\bibitem{Hadj}
N.\ Hadjisavvas, S. Koml\'{o}si, and S. Schaible, ``Handbook of generalized convexity and generalized monotonicity, nonconvex optimization and its applications," Springer-Verlag, New York, 2005.

\bibitem{John}
R.\ John, ``Uses of generalized convexity and generalized monotonicity in economics," In \cite{Hadj}, pp.~619--666.

\bibitem{maple} \texttt{http://www.maplesoft.com/}.

\bibitem{Martinez}
J.E.\ Martinez-Legaz, ``Generalized convex duality and its economic applications," In \cite{Hadj}, pp.~237--292.

\bibitem{minty}
G.J. Minty, ``On the maximal domain of a "monotone'' function," \emph{ Michigan Math. J.}~8, pp.~135--137, 1961.

\bibitem{phelps}
R.R.\ Phelps,
\emph{Convex Functions, Monotone Operators and
Differentiability},
2nd Edition, Springer-Verlag, 1993.

\bibitem{Rock70}
R.T. Rockafellar, \emph{Convex Analysis},
Princeton University Press, 1970.

\bibitem{Rockvirtual}
R.T. Rockafellar, ``On the virtual convexity of the domain and range of
a nonlinear maximal monotone operator," \emph{Math. Ann.}~185, pp.~81--90, 1970.

\bibitem{Rock98}
R.T.\ Rockafellar and R. J-B\ Wets,
\emph{Variational Analysis},
Springer-Verlag, %New York,
corrected 3rd printing, 2009.

\bibitem{Si2}
S.\ Simons, \emph{From Hahn-Banach to Monotonicity},
% Lecture Notes in Mathematics, Vol. 1693,
Springer-Verlag, 2008.


\bibitem{Zalinescu}
{C.\ Z\u{a}linescu},
\emph{Convex Analysis in General Vector Spaces}, World Scientific
Publishing, 2002.

\end{thebibliography}
\end{document}